\documentclass{amsart}

\usepackage[english]{babel}

\usepackage[letterpaper,top=2cm,bottom=2cm,left=3cm,right=3cm,marginparwidth=1.75cm]{geometry}

\usepackage{enumitem}
\usepackage{amssymb}
\usepackage{amsmath}
\usepackage{amsthm}
\usepackage{graphicx}
\usepackage{physics}
\usepackage{tikz-cd}
\usepackage{mathtools}
\usepackage[all]{xy}
\usepackage{xcolor}
\usepackage{booktabs}
\usepackage{tabularx}
\usepackage{stmaryrd}
\usepackage{appendix}

\usepackage[colorlinks=true, allcolors=blue]{hyperref}
\usepackage[nameinlink]{cleveref}

\theoremstyle{plain}
\newtheorem{theorem}{Theorem}[section]

\newtheorem{lemma}[theorem]{Lemma}
\newtheorem{proposition}[theorem]{Proposition}
\newtheorem{conjecture}[theorem]{Conjecture}
\theoremstyle{definition}
\newtheorem{definition}[theorem]{Definition}
\newtheorem{remark}[theorem]{Remark}

\newtheorem*{assumption}{Assumption}

\newcounter{assumptionref}
\newcommand{\assumptionlabel}[2]{%
  \renewcommand{\theassumptionref}{#2}%
  \refstepcounter{assumptionref}%
  \label{#1}%
}
\crefalias{assumptionref}{assumption}
\crefname{theorem}{Theorem}{Theorems}
\crefname{lemma}{Lemma}{Lemmas}
\crefname{proposition}{Proposition}{Propositions}
\crefname{corollary}{Corollary}{Corollaries}
\crefname{conjecture}{conjecture}{conjectures}
\crefname{conjecture}{Conjecture}{Conjectures}
\crefname{assumption}{assumption}{assumptions}
\crefname{assumption}{Assumption}{Assumptions}

\DeclareMathOperator{\Hom}{Hom}
\DeclareMathOperator{\depth}{depth}

\newcommand{\cD}{\mathcal D}
\newcommand{\Db}{D^b}
\newcommand{\cP}{\mathcal P}

\newcommand{\Ext}{\operatorname{Ext}}
\newcommand{\cExt}{\operatorname{\mathcal{E}xt}}
\newcommand{\cHom}{\operatorname{\mathcal{H}om}}

\newcommand{\C}{\mathbb C}
\newcommand{\PP}{\mathbb P}
\newcommand{\bbA}{\mathbb A}
\newcommand{\bbF}{\mathbb F}
\newcommand{\R}{\mathbb R}
\newcommand{\Z}{\mathbb Z}
\newcommand{\Coh}{\operatorname{Coh}}

\newcommand{\B}{\mathcal B}
\newcommand{\cA}{\mathcal A}
\newcommand{\cF}{\mathcal F}
\newcommand{\cG}{\mathcal G}
\newcommand{\cK}{\mathcal K}

\newcommand{\cO}{\mathcal O}
\newcommand{\cI}{\mathcal I}
\newcommand{\cL}{\mathcal L}

\newcommand{\fm}{\mathfrak{m}}

\newcommand{\id}{\mathrm{id}}

\newcommand{\ch}{\operatorname{ch}}

\newcommand{\im}{\operatorname{im}}
\newcommand{\Sym}{\operatorname{Sym}}

\newcommand{\ii}{\mathbf i}

\newcommand{\Pic}{\mathrm{Pic}}
\newcommand{\Perf}{\mathrm{Perf}}

\title[Indecomposability and nef canonical divisors]
{Indecomposability of derived categories and nef canonical divisors
in dimensions three and four}

\author[D. Wu]{Dongjian Wu}
\address{D. Wu: Shanghai Institute for Mathematics and Interdisciplinary Sciences (SIMIS), Shanghai 200433, China \& Research Institute of Intelligent Complex Systems, Fudan University, Shanghai 200433, China}
\email{wdj@simis.cn}

\author[N. Zhang]{Nantao Zhang}
\address{N. Zhang: School of Mathematics, Sun Yat-sen University, Guangzhou 510275, China}
\email{zhangnt@mail.sysu.edu.cn}

\subjclass[2020]{Primary 14F08, 14E30; Secondary 14J30, 14J35}
\keywords{semiorthogonal decompositions, derived categories, extremal contractions, conormal sheaves, birational geometry}

\begin{document}
\date{}
\begin{abstract}
  We prove that a smooth projective threefold or fourfold whose
bounded derived category admits no nontrivial semiorthogonal
decomposition has nef canonical divisor, which establishes
\cite[Conjecture~1.1]{KO2018} in these dimensions. The proof combines
the classification of extremal contractions with constructions of
nontrivial admissible subcategories. In dimension four, we refine
the descriptions of isolated surface fibers and construct
exceptional objects supported on certain fibers.
\end{abstract}

\maketitle

\section{Introduction}
The minimal model program is expected to be reflected in
semiorthogonal decompositions of derived categories
\cite{BO1995,Kawamata2002}. More precisely, the DK conjecture
predicts that \(K\)-equivalence between smooth projective varieties
gives rise to an equivalence of their derived categories, while
the corresponding inequality between canonical divisors gives
rise to a fully faithful embedding
\cite{Kawamata2017}. A consequence is that a smooth projective
variety whose bounded derived category admits no nontrivial
semiorthogonal decomposition has nef canonical divisor.
This statement is recorded in
\cite[Conjecture~1.1]{KO2018}. In this paper, we study it in
dimensions three and four, with particular attention to extremal
contractions and isolated surface fibers in the fourfold case.

\subsection{Background}

Let \(X\) be a smooth projective variety over \(\C\), and let \(\Db(X)\) denote its bounded derived category of coherent sheaves. The relationship between semiorthogonal decompositions of \(\Db(X)\) and birational geometry is studied in \cite{BO1995,Kawamata2002}. A basic example is the blow-up of a smooth variety along a smooth center. The blow-up formula expresses its derived category in terms of the derived categories of the original variety and the center \cite{Orlov1992}. More generally, the minimal model program is expected to give fully faithful functors and semiorthogonal decompositions governed by the canonical divisor \cite{Kawamata2002}.

The absence of semiorthogonal decompositions is also closely related to the canonical bundle. For a smooth projective curve of positive genus, \(\Db(X)\) admits no nontrivial semiorthogonal decomposition \cite{Okawa2011}. In higher dimensions, sections of the canonical bundle constrain the supports of objects in a semiorthogonal decomposition \cite{KO2018}. In particular, if the canonical base locus is finite, possibly empty, then \(\Db(X)\) admits no nontrivial semiorthogonal decomposition \cite[Corollary 1.5]{KO2018}. This criterion is refined in \cite{Lin2026} by considering the paracanonical base locus
\[\operatorname{PBs}|\omega_X|
=\bigcap_{L\in\Pic^0(X)}\operatorname{Bs}|\omega_X\otimes L|.
\]
If this locus is finite, possibly empty, then \(\Db(X)\) admits no nontrivial semiorthogonal decomposition. Among the applications are the symmetric products \(C^{(i)}\) of a smooth projective curve of genus \(g\geq2\), for \(1\leq i\leq g-1\), and certain minimal surfaces \cite{Lin2026}. The paracanonical base locus is identified with the relative base locus of the canonical bundle with respect to the Albanese morphism in \cite{Caucci2024}. Another indecomposability criterion is given in \cite{Pirozhkov2023}. This proves that varieties with finite Albanese morphism are stably semiorthogonally indecomposable. In particular, their derived categories admit no nontrivial semiorthogonal decomposition.

These results give geometric conditions under which semiorthogonal
decompositions are trivial. A complementary question asks whether
the absence of a nontrivial semiorthogonal decomposition forces
the canonical divisor to be nef. The following statement is a
consequence of the DK hypothesis \cite{BO1995,Kawamata2002}
and is recorded in \cite[Conjecture~1.1]{KO2018}.

\begin{conjecture}
    \label{conj:indecomposability-nefness}
    Let \(X\) be a smooth projective variety. If \(\Db(X)\)
    admits no nontrivial semiorthogonal decomposition, then
    \(K_X\) is nef.
\end{conjecture}

\Cref{conj:indecomposability-nefness} holds in dimension
one. In dimension two, it follows from the classification of
extremal contractions, together with the projective bundle and
blow-up formulas \cite{KM1998,Orlov1992}. The converse is false
even for surfaces. An Enriques surface has numerically trivial
canonical divisor, while its structure sheaf is exceptional and
generates a nontrivial admissible subcategory
\cite[Theorem~1.9(1)]{KO2018}.

\subsection{Main results}

Our main result is the following.

\begin{theorem}
    \label{thm:main}
    Let \(X\) be a smooth projective variety of dimension three or four. If \(\Db(X)\) admits no nontrivial semiorthogonal decomposition, then \(K_X\) is nef.
\end{theorem}

We argue by contraposition. If \(K_X\) is not nef, the cone and contraction theorems give an elementary \(K_X\)-negative contraction \(f:X\to Y\) \cite{KM1998}. We construct a nonzero proper admissible subcategory of \(\Db(X)\) for each type of contraction. When \(Y\) is smooth, the functor \(Lf^*\) is fully faithful and its image provides such a subcategory. If \(f\) is birational and its exceptional locus is an irreducible divisor contracted to a point, the structure sheaf of that divisor is exceptional by \Cref{lem:contra-point}. This assertion holds in arbitrary dimension.

In dimension three, these constructions suffice by the classification of extremal contractions \cite{Ando1985,AW1997}. In dimension four, we also use the descriptions of small and divisorial contractions in \cite{Kawamata1989,Takagi1999,AW1998}. For the divisorial contractions with singular target considered in \Cref{sec:fourfold}, we construct local categorical resolutions and use descent for semiorthogonal decompositions \cite{Kuz2008,AE2021}.

The main geometric part of the fourfold argument concerns contractions onto threefolds with isolated two-dimensional fibers. We study the types C2, C3, and C5 of \cite[Theorem 0.8]{Kachi1997} and the three-plane configuration of \cite[Proposition 4.11]{AW1998}, which we call type \(\mathrm{C}'\). Here \(\cO_E\) is viewed as a sheaf on \(X\) by extension by zero.

\begin{theorem}
    \label{thm:fibers-introduction}
    Let \(f:X^4\to Y^3\) be an elementary \(K_X\)-negative contraction from a smooth fourfold, and let \(E=(f^{-1}(P))_{\mathrm{red}}\) be an isolated surface fiber. Put \(H=-K_X\).
    \begin{enumerate}[label=(\roman*)]
        \item If \(E=A\cup_L B\) is of type C2, where \(A\cong B\cong\PP^2\), \(L\) is a line on both planes, and \(\cO_A(H)\cong\cO_A(1)\), \(\cO_B(H)\cong\cO_B(1)\), then the scheme-theoretic fiber is reduced and \(Y\) is smooth at \(P\).
        \item If \(E\cong\PP(1,1,3)\) is of type C3, with \(\cO_E(H)\cong\cO_E(3)\), then \(\cO_E\) is exceptional. The scheme-theoretic fiber is reduced and \(Y\) is smooth at \(P\).
        \item If \(E=A\cup_L B\) is of type C5, where \(A\cong\PP^2\), \(B\cong\PP(1,1,2)\), and \(\cO_A(H)\cong\cO_A(1)\), \(\cO_B(H)\cong\cO_B(2)\), then \(\cO_E\) is exceptional.
        \item If \(E\) is of type \(\mathrm{C}'\), with three planes \(A_1,B,A_2\) such that \(A_i\cap B=L_i\) is a line on both planes and \(A_1\cap A_2=L_1\cap L_2\) is a point, then \(\cO_E\) is exceptional.
    \end{enumerate}
\end{theorem}

The assertions are proved in \Cref{thm:c2}, \Cref{thm:c3-exceptional}, \Cref{thm:c3-smooth}, \Cref{thm:c5}, and \Cref{thm:three-planes}. The proofs study the conormal sheaves of the reduced fibers, their restrictions to anticanonical slices, and the blow-ups along the fibers. The conormal descriptions give the required self-Ext vanishings in cases C3, C5, and \(\mathrm{C}'\). They also determine the scheme-theoretic fibers and the local structure of the target in the C2 and C3 cases. Together with the remaining contraction types, these results prove \Cref{thm:main}. We construct an explicit example of a C5 fiber in \Cref{app:c5}.

\subsection{Contents}

In \Cref{sec:background}, we recall the categorical and birational facts used in the proof. In \Cref{sec:threefold}, we treat the threefold case and prove the general lemma for divisors contracted to a point. In \Cref{sec:refinement}, we study the special fibers of types C2, C3, C5, and \(\mathrm{C}'\). In \Cref{sec:fourfold}, we combine these results with the classification of extremal contractions to complete the fourfold argument. \Cref{app:c5} gives an example of a C5 fiber.

\subsection*{Conventions} Throughout the paper, the base field is \(\C\). All varieties are projective unless otherwise stated.
For a morphism \(f\), we write \(f_*\) and \(f^*\) for the underived functors on sheaves, and \(Rf_*\) and \(Lf^*\) for their total derived functors. We retain the symbols \(R\) and \(L\) for functors between derived categories, even when the corresponding functor on sheaves is exact. We denote higher direct-image sheaves by \(R^qf_*\) and use the usual notation \(f^*\) for pullbacks of functions and divisors.

A contraction is a projective surjective morphism with connected fibers onto a normal variety. We call it \emph{good} if its source is smooth and its anticanonical divisor is relatively ample, and \emph{crepant} if its canonical divisor is relatively numerically trivial. Statements near a fiber are understood after restricting to a sufficiently small neighborhood of its image.

For a closed subscheme \(Z\subset X\), we denote its ideal sheaf by \(\cI_{Z/X}\) and its normal sheaf by \(N_{Z/X}=\cHom_Z(\cI_{Z/X}/\cI_{Z/X}^2,\cO_Z)\). We use the terms normal bundle and conormal bundle when the embedding is regular. For a coherent sheaf \(\cF\) on an integral scheme, write \(\cF^{\mathrm{tf}}=\cF/\operatorname{tors}(\cF)\). The projectivization \(\PP_Z(\cF)\) parametrizes rank-one quotients, with tautological quotient line bundle \(\cO_{\PP_Z(\cF)}(1)\). On \(S_m=\PP(1,1,m)\), \(\cO_{S_m}(d)\) denotes the sheaf associated with the graded twist of degree \(d\).

\subsection*{Acknowledgements}

The authors would like to thank Will Donovan, Yuwei Fan, Xun Lin,
and Shizhuo Zhang for helpful discussions, and Kotaro Kawatani
and Shinnosuke Okawa for their comments.

\subsection*{AI disclosure}

The authors proved the threefold case of \Cref{thm:main} and developed a
strategy for the fourfold case. ChatGPT 6.0 was used to fill in some gaps. The
authors reviewed the proof, corrected errors in the original version, and
reorganized it into its current form. ChatGPT 6.0 was also used to correct typos
and polish the paper. ChatGPT 6.2 Sol was used to discover the missing
\(\mathrm{C}'\) case in the original proof.

\section{Background on derived categories and birational geometry}
\label{sec:background}

In this section, we recall the definitions and basic properties of
semiorthogonal decompositions, admissible subcategories, and exceptional
objects, following \cite{BO1995,Huybrechts2006,KO2018}. We also recall
the relevant results on extremal contractions and categorical constructions
from \cite{KM1998,Orlov1992}.

\subsection{Derived categories and semiorthogonal decompositions}

For a variety \(X\), let \(\Db(X)=D^b(\Coh(X))\) be the bounded derived category of coherent sheaves on \(X\). We write
\[\Ext_X^k(F,G)=\Hom_{\Db(X)}(F,G[k]).\]
For smooth projective \(X\), Serre duality takes the form
\[\Hom_{\Db(X)}(F,G)\cong
\Hom_{\Db(X)}(G,F\otimes\omega_X[\dim X])^\vee;
\]
thus \(F\mapsto F\otimes\omega_X[\dim X]\) is the Serre functor \cite[Section 2]{BO1995}.

\begin{definition}
    A \emph{semiorthogonal decomposition} \(\Db(X)=\langle\cA,\B\rangle\) consists of full triangulated subcategories such that \(\Hom(B,A[k])=0\) for all \(A\in\cA\), \(B\in\B\), and \(k\in\Z\), and every \(F\in\Db(X)\) fits into a distinguished triangle
    \[B\longrightarrow F\longrightarrow A\longrightarrow B[1],
    \qquad A\in\cA,\quad B\in\B.
    \]
    It is \emph{nontrivial} if both components are nonzero. A full triangulated subcategory is \emph{admissible} if its inclusion has both a left and a right adjoint.
\end{definition}

For an admissible subcategory \(\cA\subset\Db(X)\), let
\begin{align*}
\cA^\perp&=\{F:\Hom(A,F[k])=0\text{ for all }A\in\cA,\ k\in\Z\},\\
{}^\perp\cA&=\{F:\Hom(F,A[k])=0\text{ for all }A\in\cA,\ k\in\Z\}.
\end{align*}
The adjunction triangles give
\[\Db(X)=\langle\cA^\perp,\cA\rangle
=\langle\cA,{}^\perp\cA\rangle.
\] Consequently a nonzero proper admissible subcategory suffices to prove the existence of a nontrivial semiorthogonal decomposition. An object \(E\in\Db(X)\) is \emph{exceptional} if \(\Hom(E,E)=\C\) and \(\Ext_X^k(E,E)=0\) for every \(k\neq0\), or equivalently if \(R\Hom_X(E,E)\cong\C\).

\begin{proposition}[{\cite[Lemma 1.58]{Huybrechts2006}}]
    \label{prop:exceptional-background}
    Let \(X\) be a smooth projective variety of positive dimension. An exceptional object \(E\) generates a nonzero proper admissible subcategory \(\langle E\rangle\subset\Db(X)\). In particular, \(\Db(X)\) admits a nontrivial semiorthogonal decomposition.
\end{proposition}

A closed immersion \(i:Z\hookrightarrow X\) is \emph{regular} of codimension \(r\) if its ideal is locally generated by a regular sequence of length \(r\). Its conormal bundle \(\cI_{Z/X}/\cI_{Z/X}^2\) is then locally free of rank \(r\), and its dual is \(N_{Z/X}\).

\begin{proposition}
    \label{prop:koszul-background}
    For a regular closed immersion \(i:Z\hookrightarrow X\) of codimension \(r\) into a smooth projective variety, there is a spectral sequence
    \[E_2^{p,q}=H^p(Z,\wedge^qN_{Z/X})
    \ \Longrightarrow\ \Ext_X^{p+q}(i_*\cO_Z,i_*\cO_Z).
    \]
    If \(H^0(Z,\cO_Z)=\C\) and all other terms on the \(E_2\)-page vanish, then \(i_*\cO_Z\) is exceptional.
\end{proposition}

\begin{proof}
    Applying \(\cHom_X(-,i_*\cO_Z)\) to a local Koszul resolution gives
    \[\cExt_X^q(i_*\cO_Z,i_*\cO_Z)
    \cong i_*(\wedge^qN_{Z/X}),\qquad\text{ for } 0\leq q\leq r,
    \]
    with all other sheaf Ext groups zero. The local-to-global Ext spectral sequence then gives the assertion. Under the stated vanishing, its only surviving term is \(\C\) in degree zero.
\end{proof}

\subsection{Extremal contractions and categorical constructions}

A divisor \(D\) on \(X\) is \emph{nef} if \(D\cdot C\geq0\) for every integral curve \(C\subset X\). Let \(N_1(X)_\R\) be the real vector space of curve classes modulo numerical equivalence, and let \(\overline{\operatorname{NE}}(X)\) be the closed cone generated by effective curves. A ray \(\mathcal R\) of this cone is called \emph{extremal} if \(u+v\in\mathcal R\), with \(u,v\) in the cone, implies \(u,v\in\mathcal R\). It is \(K_X\)-negative if \(K_X\cdot u<0\) for every nonzero \(u\in\mathcal R\).

\begin{proposition}[{\cite[Theorem 3.7]{KM1998}}]
    \label{prop:contraction-background}
    If \(X\) is smooth projective and \(K_X\) is not nef, there is a \(K_X\)-negative extremal ray \(\mathcal R\) and a contraction \(f:X\to Y\) such that an integral curve is contracted precisely when its numerical class lies in \(\mathcal R\). Moreover, \(-K_X\) is \(f\)-ample and the relative Picard number is \(\rho(X/Y)=1\).
\end{proposition}

Here, the relative Picard number \(\rho(X/Y)\) is given by the dimension of the space of real Cartier divisor classes modulo numerical equivalence on contracted curves. A contraction with relative Picard number one is called \emph{elementary}. If \(\dim Y<\dim X\), it is of \emph{fiber type} and an elementary \(K_X\)-negative contraction of fiber type is a Mori fiber space. Otherwise it is birational, and is called \emph{divisorial} if its exceptional locus is a divisor and \emph{small} if that locus has codimension at least two. The exceptional locus \(\operatorname{Exc}(f)\) is the locus where \(f\) is not an isomorphism. These notions organize the cases considered below. When the target is smooth, the contraction itself provides the required admissible subcategory.

\begin{lemma}
    \label{lem:smooth}
    Let \(f:X\to Y\) be a nonisomorphic \(K_X\)-negative contraction between smooth projective varieties. Then \(Lf^*\Db(Y)\) is a nonzero proper admissible subcategory of \(\Db(X)\).
\end{lemma}

\begin{proof}
    Relative Kawamata--Viehweg vanishing gives
    \(Rf_*\cO_X\cong\cO_Y.\) Hence
    the projection formula shows that the unit \(F\to Rf_*Lf^*F\) is an isomorphism, so \(Lf^*\) is fully faithful. Its right adjoint is \(Rf_*\) and Serre duality supplies a left adjoint. Its image is therefore admissible and contains
    \(\cO_X\).

   Since \(f\) is not an isomorphism, Zariski's main theorem gives
    a positive-dimensional fiber \(f^{-1}(y)\). Choose
    \(x\in f^{-1}(y)\) and consider the counit triangle
    \[
    Lf^*\cO_y\longrightarrow\cO_x\longrightarrow C\longrightarrow Lf^*\cO_y[1].
    \]
    Then
    \[
    Rf_*C=0,\qquad
    \mathcal H^0(Lf^*\cO_y)=\cO_{f^{-1}(y)}\not\cong\cO_x,
    \]
    so \(C\neq0\). By adjunction,
    \[
    \Hom_X(Lf^*F,C[k])
    \cong\Hom_Y(F,Rf_*C[k])=0
    \]
    for all \(F\in\Db(Y)\) and \(k\in\Z\). Thus the image is proper.
\end{proof}

Let \(f:X=\operatorname{Bl}_Z Y\to Y\) be the blow-up of a smooth projective variety along a smooth center of codimension \(c\geq2\), and let \(i:E\hookrightarrow X\) and \(p:E\to Z\) be the exceptional divisor and its projection. Write \(\cO_E(-1)=\cO_E(E)\), and set \(\Phi_k(F)=Ri_*(Lp^*F\otimes\cO_E(k))\). Then, \cite{Orlov1992} gives
\[\Db(X)=\big\langle
\Phi_{1-c}\Db(Z),\ldots,\Phi_{-1}\Db(Z),Lf^*\Db(Y)
\big\rangle.
\]
For the blow-up of a surface at a point, the additional component is generated by the exceptional object \(i_*\cO_E(-1)\). In the following
sections, we construct admissible subcategories similarly by pullback or by
finding exceptional objects supported on the contracted loci for higher dimension varieties.

\section{The threefold case}
\label{sec:threefold}

In this section, we treat the threefold case by studying
\(K_X\)-negative extremal contractions. We use pullback when the target
is smooth and show that the structure sheaf of the exceptional divisor
is exceptional when it is contracted to a point.

Let \(X\) be a smooth projective threefold. If \(K_X\) is not nef, by \Cref{prop:contraction-background}, there are a \(K_X\)-negative extremal ray and a contraction
\(f: X \to Y\)
where \(-K_X\) is \(f\)-ample. By the classification of extremal contractions from smooth threefolds \cite{AW1997}, either the target is smooth or the contraction sends an irreducible divisor to a point.

\begin{proposition}
    If \(X\to Y\) is a Mori fiber space or contracts the exceptional divisor to a curve, then \(\Db(X)\) admits a nontrivial semiorthogonal decomposition.
\end{proposition}

\begin{proof}
    A Mori fiber space from a smooth projective threefold has a smooth target and contraction to a curve is a smooth blow-up \cite{AW1997}. The result therefore follows from \Cref{lem:smooth}.
\end{proof}

\begin{lemma}
\label{lem:deduce_exceptional}
    If a divisorial extremal contraction sends its exceptional divisor \(E\) to a point, then \(\cO_E\) is exceptional.
\end{lemma}

\begin{proof}
    This is the three-dimensional instance of the following general lemma, \Cref{lem:contra-point}.
\end{proof}

The preceding statement is a special case of the following general lemma.
\begin{lemma}
    \label{lem:contra-point}
    Let \(X\) be a smooth projective variety, and let \(f:X\to Y\) be an elementary birational \(K_X\)-negative contraction whose irreducible exceptional divisor \(E\) maps to a point. Then \(\cO_E\) is exceptional.
\end{lemma}

\begin{proof}
    Put \(n=\dim X\) and \(P=f(E)\). Since \(f\) is elementary,
    the negativity lemma implies that \(-E\) is \(f\)-ample.
    Relative Kawamata--Viehweg vanishing and
    Grauert--Riemenschneider vanishing
    \cite[Corollary 2.68]{KM1998} give
    \[
    R^qf_*\cO_X(-jE)=R^qf_*\omega_X(-jE)=0,
    \qquad q>0,\quad j=0,1.
    \]
    Applying \(Rf_*\) to
    \[
    0\to\cO_X(-E)\to\cO_X\to\cO_E\to0,
    \qquad
    0\to\omega_X(-E)\to\omega_X\to\omega_X|_E\to0
    \]
    yields
    \[
    H^q(E,\cO_E)=H^q(E,\omega_X|_E)=0,
    \qquad q>0.
    \]
    Moreover, since \(E\) is integral and projective and
    \(\omega_X|_E\) is antiample, we obtain
    \(
    H^0(E,\cO_E)=\C,\) and \(H^0(E,\omega_X|_E)=0.
    \)
    Thus
    \(
    R\Gamma(E,\cO_E)\cong\C,\) and \(
    R\Gamma(E,\omega_X|_E)=0.
    \)
    The Cartier divisor \(E\) is Gorenstein. Adjunction and
    Serre duality give
    \[
    H^q(E,\cO_E(E))^\vee
    \cong H^{n-1-q}(E,\omega_E(-E))
    \cong H^{n-1-q}(E,\omega_X|_E)=0.
    \]
    The Cartier self-intersection calculation therefore gives
    \(
    R\Hom_X(\cO_E,\cO_E)\cong\C.
    \)
\end{proof}

\begin{remark}
    \Cref{lem:smooth,lem:contra-point} give
    \Cref{conj:indecomposability-nefness} for surfaces.
    In dimension three, the classification of extremal contractions
    and the preceding lemmas give a nonzero proper admissible
    subcategory of \(\Db(X)\) whenever \(K_X\) is not nef.
    This proves the threefold case of \Cref{thm:main}.
\end{remark}

\section{Refinement of Kachi and Andreatta--Wi\'sniewski}
\label{sec:refinement}

In this section, we refine the descriptions of isolated surface fibers
of elementary \(K_X\)-negative contractions \(f:X\to Y\) from smooth
fourfolds to threefolds given in \cite{Kachi1997,AW1998}. The conormal
calculations give smoothness of the target in cases C2 and C3 and
exceptional objects in cases C3, C5, and \(\mathrm{C}'\).

For simplicity, we write
\(S_m=\PP(1,1,m)\). The cases of length two or analytic relative Picard number two are given by \cite[Theorems 0.5--0.7]{Kachi1997}. For length one and analytic relative Picard number one, we use \cite[Propositions 10.6--10.7]{Kachi1997} together with \cite[Proposition 4.11]{AW1998}.

Relative generation of \(\cO_X(-K_X)\) and \cite[Proposition 4.2.1]{AW1998} imply that every irreducible component of a reduced surface fiber is normal. A normal cone \(S_m\) has embedding dimension \(m+1\) at its vertex, so \(m\leq3\). Type C4 is excluded by \cite[Lemma 4.10.2]{AW1998}, since two components that meet along a line cannot both be cones. For a plane or an irreducible quadric surface, smoothness of the target follows from \cite[Theorem 0.6]{Kachi1997} or \cite[Theorem 5.9.6]{AW1998}.
The resulting cases are summarized in \Cref{tab:special-fibers}.
\begin{table}[ht]
\centering
\begin{minipage}{\textwidth}
\caption[Classification of special fibers]{Classification of special fibers\protect\footnote{%
    For type \(\mathrm{C}'\), Kachi's preprint claims that a fiber whose components are all planes has at most two
    components, arguing that three concurrent lines in an anticanonical slice
    contradict \(H^1(\cO)=0\) \cite[(10.6.8)]{Kachi1995}.
    In the published version, (10.6.8) raises the bound to three, and
    Proposition~10.6(2) includes the three-plane case.
    However Theorems~10.1(2) and~0.8(C-2) retain the two-plane statement \cite{Kachi1997},
    while Andreatta and Wi\'sniewski list the three-plane configuration explicitly
    \cite[Proposition 4.11]{AW1998}.}}
\label{tab:special-fibers}
\begin{tabularx}{\linewidth}{c c X}
\toprule
\textbf{Type} & \textbf{Fiber} & \textbf{Geometric conclusion} \\
\midrule
A & \(\PP^2\) & smooth target by \cite[Theorem 0.6]{Kachi1997}. \\
B0 & \(E = A \cup B, A \cong B \cong \PP^2, A \cap B = \{\mathrm{pt}\}\) & \(N_{A/X} \cong N_{B/X} \cong \cO_{\PP^2}(-1)^{\oplus 2}\) by \cite[Theorem 0.7]{Kachi1997}. \\
B1--B5 & several cases & smooth target by \cite[Theorem 0.7]{Kachi1997}. \\
C1 & \(\PP^2\) & smooth target by \cite[Theorem 5.9.6]{AW1998}. \\
C2 & \(E=A\cup_L B,\ A\cong B\cong\PP^2,\ L\cong\PP^1\) & \\
C3 & \(E\cong S_m\), \(m=2,3\) & smooth target for \(m=2\) by \cite[Theorem 5.9.6]{AW1998}. \\
C4 & union of cones & impossible by \cite[Lemma 4.10.2]{AW1998}. \\
C5 & \(E=A\cup_L B,\ A\cong\PP^2,\ B\cong S_2,\ L\cong\PP^1\) & - \\
\(\mathrm{C}'\) & \(\begin{gathered}
    E=A_1\cup B\cup A_2,\quad A_1,B,A_2\cong\PP^2,\\
    A_i\cap B=L_i, \quad A_1\cap A_2=L_1\cap L_2=\{v\},
\end{gathered}\) & - \\
\bottomrule
\end{tabularx}
\end{minipage}
\end{table}

In case B0, both \(\cO_A\) and \(\cO_B\) are exceptional. When the target is smooth, pullback supplies an admissible subcategory. The remaining cases are C2, C3 with \(m=3\), C5, and \(\mathrm{C}'\).

\subsection{The Andreatta--Wi\'sniewski method}
We will use cohomological criteria and functions on the Stein factors of
anticanonical slices in the deformation argument of \Cref{prop:non-neg}
to exclude negative quotients of the conormal sheaf on lines, and then
give a criterion for reducedness of the scheme-theoretic fiber.

Let \(f:X\to Y\) be an elementary \(K_X\)-negative contraction
from a smooth fourfold to a threefold. After replacing \(Y\)
by an affine neighborhood of \(P\), we assume that \(f^{-1}(P)\)
is the only two-dimensional fiber. Set
\[
E=(f^{-1}(P))_{\mathrm{red}},\qquad
\cI=\cI_{E/X},\qquad E_n=V(\cI^n),
\]
\[
M=\cI/\cI^2,\qquad
N=\cHom_E(M,\cO_E),\qquad H=-K_X.
\]
Thus \(E_1=E\), and \(E_2\) is its first infinitesimal
neighborhood. The scheme-theoretic fiber has ideal
\(J=\fm_P\cO_X\). Since
\(
J\subseteq\cI=\sqrt{J},\) and \(\cI^c\subseteq J
\)
for some \(c>0\), the \(J\)-adic and \(\cI\)-adic systems are
cofinal. We may therefore use the neighborhoods \(E_n\) in
formal functions before proving that \(J=\cI\).

\subsubsection{Cohomology of the reduced fiber}

\begin{assumption}[A1]
\assumptionlabel{ass:A1}{A1}
The embedding \(E\subset X\) is Cohen--Macaulay of codimension
two. The polarization \(H|_E\) identifies \(E\) with one of
the following surfaces.
\begin{enumerate}[label=(\roman*)]
    \item A reduced union of two planes in \(\PP^3\), with
    its hyperplane bundle.
    \item A degree-three surface in \(\PP^4\) with resolution
    \[
    0\to\cO_{\PP^4}(-3)^2
    \to\cO_{\PP^4}(-2)^3
    \to\cO_{\PP^4}\to\cO_E\to0.
    \]
\end{enumerate}
\end{assumption}

In either case,
\begin{equation}
\label{eq:vanishing}
R\Gamma(E,\cO_E)\cong\C,\qquad
R\Gamma(E,\cO_E(-H))=0,\qquad
R\Gamma(E,\omega_E(H))=0.
\end{equation}
For the union of two planes, these follow from
\[
0\to\cO_{\PP^3}(-2)\to\cO_{\PP^3}\to\cO_E\to0
\]
and adjunction \(\omega_E\cong\cO_E(-2H)\).
For the surface of degree three, the resolution in
\Cref{ass:A1} gives the first two vanishings.
Dualizing with \(\omega_{\PP^4}=\cO_{\PP^4}(-5)\) and
twisting by \(H\) gives
\[
0\to\cO_{\PP^4}(-4)
\to\cO_{\PP^4}(-2)^3
\to\cO_{\PP^4}(-1)^2
\to\omega_E(H)\to0,
\]
which gives the third vanishing.

The following lemma reduces exceptionality to the cohomology
of the normal sheaf.

\begin{lemma}
\label{lem:c-exc}
Let \(E\subset X^4\) be Cohen--Macaulay of codimension two,
and put
\[
N=\cHom_E(\cI_{E/X}/\cI_{E/X}^2,\cO_E),
\qquad
\cG=\omega_E\otimes\omega_X^{-1}|_E.
\]
If \(R\Gamma(E,\cO_E)\cong\C\) and \(R\Gamma(E,\cG)=0\), then
\[
\Ext_X^{j+1}(\cO_E,\cO_E)\cong H^j(E,N),
\qquad 0\leq j\leq2.
\]
If also \(R\Gamma(E,N)=0\), then \(\cO_E\) is exceptional.
\end{lemma}

\begin{proof}
Since \(X\) is smooth and \(E\) is Cohen--Macaulay of
codimension two, \(\cO_E\) has a locally free resolution of
length two. It gives
\[
\cExt_X^q(\cO_E,\cO_E)\cong
\begin{cases}
\cO_E & q=0,\\
N & q=1,\\
\cG & q=2,\\
0 & q>2.
\end{cases}
\]
The assertion follows from the local-to-global Ext spectral
sequence and the assumed vanishings.
\end{proof}

Under \Cref{ass:A1}, one has
\(\cG=\omega_E(H)\), so the cohomological hypotheses of
\Cref{lem:c-exc} follow from \eqref{eq:vanishing}.
We will also use the following deformation bound.

\begin{lemma}
\label{lem:deform}
Let \(E\subset X\) be the reduced support of an isolated surface fiber.
Suppose that \(E\) is Cohen--Macaulay of codimension two and satisfies
the cohomological hypotheses of \Cref{lem:c-exc}. Then
\[
h^0(E,N)-h^1(E,N)\leq0.
\]
\end{lemma}

\begin{proof}
Consider a small extension \(R'\twoheadrightarrow R\) of local Artinian
\(\C\)-algebras with kernel \(\mathfrak a\) and an embedded deformation
of \(E\) over \(R\). By Schaps's theorem, the deformation has local
Hilbert--Burch presentations and lifting their matrices gives flat
embedded lifts over \(R'\) \cite{Schaps1977}. On overlaps, two lifts differ by
a section of \(N\otimes_\C\mathfrak a\). These differences form a
\v{C}ech cocycle whose class in \(H^1(E,N)\otimes_\C\mathfrak a\)
vanishes precisely when the local lifts can be adjusted to glue.
Thus the Hilbert functor has tangent space \(H^0(E,N)\) and
obstruction space \(H^1(E,N)\).

The tangent--obstruction estimate gives
\[
\dim_{[E]}\operatorname{Hilb}(X)
\geq h^0(E,N)-h^1(E,N).
\]
If the right-hand side were positive, there would be a
nonconstant Hilbert family through \(E\). Nearby members
are proper reduced connected surfaces with \(h^0(\cO)=1\).
Their maps to the affine base are constant, and isolation
of the surface fiber forces them into \(f^{-1}(P)\).
They are therefore reduced subschemes of \(E\) with the
same Hilbert polynomial as \(E\), and hence equal to \(E\).
This is a contradiction.
\end{proof}

\subsubsection{The blow-up and anticanonical slices}

Consider the blow-up
\[
\beta:T=\operatorname{Bl}_E X\to X,\qquad
\cI\cO_T=\cO_T(-D)=\cL.
\]

\begin{assumption}[A2]
\assumptionlabel{ass:A2}{A2}
The fourfold \(T\) is normal and Gorenstein, and
\[
T=\PP_X(\cI),\qquad
D=\PP_E(M),\qquad
\cL|_D\cong\cO_D(1),\qquad
K_T=\beta^*K_X+D.
\]
\end{assumption}

A quotient \(M|_l\twoheadrightarrow\cO_l(-1)\) on an
\(H\)-degree-one line determines a curve
\(\gamma\cong l\) in \(D\), with
\[
\cL|_\gamma\cong\cO_l(-1),\qquad
D\cdot\gamma=\beta^*H\cdot\gamma=1,\qquad
K_T\cdot\gamma=0.
\]
The last equality is the starting point of the deformation
argument.

We next restrict the conormal sheaf to an anticanonical slice.
The relative anticanonical bundle is generated by sections
by \cite[Theorem 0.5(1)]{Kachi1997}. Choose a general member
\(D_h=(h=0)\in|H|\).

\begin{assumption}[A3]
\assumptionlabel{ass:A3}{A3}
The polarization \(H|_E\) is very ample, and \(D_h\) is smooth
near \(E\). The curve \(C=E\cap D_h\) is a chain of at most three
copies of \(\PP^1\), with consecutive components meeting transversally
at one point. The embedding
\(E\subset X\) is regular near \(C\), and \(h\) is a
nonzerodivisor on both \(\cO_E\) and \(\cO_{E_2}\).
These properties hold on an open set of slice parameters.
\end{assumption}

The target of the elementary contraction is Gorenstein near \(P\) by
\cite[Section 5.9]{AW1998}. Relative vanishing and
duality give
\begin{equation}
\label{eq:aw}
Rf_*\cO_X\cong\cO_Y,\qquad
Rf_*\omega_X\cong\omega_Y[-1]
\end{equation}
by \cite[Theorem 1.2.4]{AW1998}.
Write \(g=f|_{D_h}\), and take its Stein factorization
\[
D_h\xrightarrow{\psi}Z\xrightarrow{\tau}Y.
\]
A general fiber \(F_y\) of \(f\) is a smooth conic, so
\[
\operatorname{length}(D_h\cap F_y)
=(-K_X)\cdot F_y=2.
\]
Thus \(\tau\) has degree two. Applying \(Rf_*\) to the
anticanonical divisor sequence gives
\[
0\to\cO_Y\to\tau_*\cO_Z\to\omega_Y\to0.
\]
In particular, \(\tau_*\cO_Z\) is locally free of rank two.
After shrinking \(Y\), the trace splitting gives
\[
\tau_*\cO_Z=\cO_Y\oplus\cO_Yr,\qquad
r^2=b\in\cO_Y.
\]

The reduced support of \(g^{-1}(P)\) is the connected curve
\(C\). Hence \(\tau^{-1}(P)\) has a single underlying point,
denoted by \(z_P\). Its fiber algebra is
\(\C[r]/(r^2-b(P))\), so \(b(P)=0\). Thus
\[
b\in\fm_P,\qquad r(z_P)=0.
\]
The pullback of \(r\) vanishes on the reduced curve \(C\).
We denote its first-order conormal class by \(\bar r\).

We will use the following vanishing result
\cite[Lemma 1.2.2(b)]{AW1998}.
This result is independent of \Cref{ass:A1}--\Cref{ass:A3}.

\begin{lemma}
\label{lem:aw-slice-vanishing}
Let \(f:X\to Y\) be a good contraction with
\(\dim X-\dim Y=1\), and let
\(E=(f^{-1}(P))_{\mathrm{red}}\) be a two-dimensional fiber.
If \(F'\subset X\) is supported on \(E\) and
\(D\in|-K_X|\) contains no irreducible component of \(E\), then
\[
H^1(F'\cap D,\cO_{F'\cap D})=0.
\]
In particular, this applies to \(F'=E_n\) for every \(n\geq1\).
\end{lemma}

\begin{assumption}[A4]
\assumptionlabel{ass:A4}{A4}
For every \(H\)-degree-one line \(l\subset E\), there is
an anticanonical slice \(D_{h'}\) containing \(l\), where
\(h'\) avoids all components of \(E\) and is a nonzerodivisor
on \(\cO_{E_2}\). For \(C'=E\cap D_{h'}\), one has
\[
H^0(C',\cO_{C'})=\C,\qquad
H^1(E_2\cap D_{h'},\cO_{E_2\cap D_{h'}})=0.
\]
The same cohomology assertions hold for a general slice
in \Cref{ass:A3}.
\end{assumption}

The nonzerodivisor condition identifies the conormal bundle
of \(C\subset D_h\) with \(M|_C\), and gives
\[
0\to M|_C\to\cO_{E_2\cap D_h}\to\cO_C\to0.
\]
Since constants lift and
\(H^1(\cO_{E_2\cap D_h})=0\), we obtain
\[
H^1(C,M|_C)=0.
\]

\subsubsection{First-order sections}

We record the uniqueness condition needed to pass from
deformations of quotient curves to moving lines on \(E\).

\begin{assumption}[A5]
\assumptionlabel{ass:A5}{A5}
For every \(H\)-degree-one line \(l\subset E\), there is at
most one surjection \(M|_l\twoheadrightarrow\cO_l(-1)\),
up to an automorphism of the target.
\end{assumption}

Functions vanishing at \(P\) give first-order conormal
sections
\[
\theta:\fm_P\to H^0(E,M),\qquad
a\longmapsto[f^*a].
\]
Put
\[
W=\im\theta,\qquad
W_C=\im\bigl(W\to H^0(C,M|_C)\bigr).
\]
To compare \(W_C\) with all conormal sections on \(C\),
we recall the lifting result
\cite[Lemma 5.3(b)]{AW1998}.

\begin{lemma}
\label{lem:aw-conormal-lifting}
Let \(f:X\to Y\) be a good contraction or a crepant
birational contraction, with \(X\) smooth. Suppose that
\(E=(f^{-1}(P))_{\mathrm{red}}\) is locally complete
intersection and \(\operatorname{Bl}_E X\) is klt.
If \(M=\cI_{E/X}/\cI_{E/X}^2\) is nef, then
\[
\fm_P\to H^0(E,M),\qquad a\longmapsto[f^*a],
\]
is surjective.
\end{lemma}

For crepant contractions of smooth threefolds, the relevant
conormal bundles for fibers with one or two components are described in
\cite[Propositions 5.6.2--5.6.3]{AW1998}. In the splitting cases below,
sections glue along the chain, and the conormal bundle is globally
generated. The blow-up is smooth away from one ordinary double point
over each node. The lifting lemma then gives surjectivity from functions
on the Stein factor.

\begin{assumption}[A6]
\assumptionlabel{ass:A6}{A6}
On a general slice satisfying \Cref{ass:A3},
the conormal bundle has one of the following forms.
\begin{enumerate}[label=(\roman*)]
    \item \(C\cong\PP^1\), and
    \(M|_C\cong\cO(1)^{\oplus 2}\) or
    \(\cO\oplus\cO(2)\).
    \item \(C=C_A\cup C_B\), where the components are copies
    of \(\PP^1\) meeting transversally at one point, and
    \(M|_{C_i}\cong\cO\oplus\cO(1)\).
    \item \(C\) is a chain with three components. The conormal restricts
    to \(\cO^{\oplus 2}\) on the middle component and to
    \(\cO\oplus\cO(1)\) on each end component.
\end{enumerate}
Moreover, the first-order map
\[
\fm_{z_P}\to H^0(C,M|_C)
\]
from functions on \(Z\) is surjective.
\end{assumption}

Every function on \(Z\) vanishing at \(z_P\) can be written
as \(a+cr\), with \(a\in\fm_P\). Since \(r\) vanishes on
\(C\), its first-order class depends only on the value of
\(c\) at \(P\). Consequently
\begin{equation}
\label{eq:W}
H^0(C,M|_C)=W_C+\C\bar r.
\end{equation}
In all cases of \Cref{ass:A6},
\[
h^0(C,M|_C)=4,\qquad \dim W_C\geq3.
\]
Restriction from \(H^0(C,M|_C)\) to any component is surjective,
since any section on that component extends along the chain by
matching values at the nodes. Consequently \eqref{eq:W} gives
\[
\dim(W_C|_{C_i})\geq h^0(C_i,M|_{C_i})-1.
\]

The next assumption makes restriction of global conormal
sections to the slice injective.

\begin{assumption}[A7]
\assumptionlabel{ass:A7}{A7}
\(H^0(E,M(-H))=0\).
\end{assumption}

\begin{lemma}
\label{lem:fb}
Suppose that \(h\) is a nonzerodivisor on
\(\cO_X/\cI^2\) and \Cref{ass:A7} holds.
Then \(f^*b\in\cI^2\). Its quadratic initial class on
the slice is \(\bar r^2\).
\end{lemma}

\begin{proof}
The slice hypotheses give
\[
0\to M(-H)\xrightarrow{\cdot h}M\to M|_C\to0.
\]
\Cref{ass:A7} therefore makes
\(H^0(E,M)\to H^0(C,M|_C)\) injective.
Since \(b\in\fm_P\), its pullback belongs to \(\cI\).
On \(D_h\), one has \(f^*b=r^2\), whose first-order class
along \(C\) is zero. Thus
\(
[f^*b]=0\) in $H^0(E,M)$,
and \(f^*b\in\cI^2\). Regularity of the embedding near
\(C\) identifies its quadratic initial class with
\(\bar r^2\).
\end{proof}

\subsubsection{Excluding negative quotients}

We first record the elementary fact about sections on
\(\PP^1\) used in the argument. For a rank-two vector
bundle \(V\) on \(\PP^1\) and
\(U\subset H^0(\PP^1,V)\), a horizontal base curve is a
curve in \(\PP(V)\) dominating \(\PP^1\) on which all
tautological sections induced by \(U\) vanish.

\begin{lemma}
\label{lem:horizontal-base}
A subspace of codimension at most one in
\(H^0(\cO(1)^{\oplus 2})\) has no horizontal base curve.
For \(V=\cO\oplus\cO(d)\), with \(0\leq d\leq2\), suppose
\(U\subset H^0(V)\) has dimension at least \(d+1\) and has
a horizontal base curve. Then \(U\) consists of the sections
of an \(\cO(d)\) summand. The base curve is the section of
\(\PP(V)\) given by the quotient to \(\cO\), and every section
outside \(U\) is nowhere zero on it.
\end{lemma}

\begin{proof}
A horizontal base curve implies that the evaluation map
\(U\otimes\cO_{\PP^1}\to V\) has generic rank one.
Let \(L\subset V\) be the saturation of its image.
Then \(U\subset H^0(L)\).

If \(V=\cO(1)^{\oplus 2}\), every line subbundle has degree
at most one. Thus \(h^0(L)\leq2\), whereas \(\dim U\geq3\).

If \(V=\cO\oplus\cO(d)\), a line subbundle has degree at most \(d\),
and one of degree at most \(d-1\) has at most \(d\) sections. Hence
\(L\cong\cO(d)\) and \(U=H^0(L)\). The quotient \(V/L\) is \(\cO\).
A complementary section has a nonzero constant component,
so its restriction to this quotient section is nowhere zero.
\end{proof}

To apply deformation theory when \(T\) is singular, we
require the following resolution.

\begin{assumption}[A8]
\assumptionlabel{ass:A8}{A8}
There is a projective small crepant resolution
\(\mu:\tilde T\to T\), with \(\tilde T\) smooth, which is
an isomorphism over the smooth locus. Every degree-\(-1\)
quotient curve over an \(H\)-degree-one line has a unique
lift \(\tilde\gamma\), mapped isomorphically onto \(\gamma\).
\end{assumption}

When \(T\) is smooth, we take \(\mu=\id_T\).

\begin{proposition}
\label{prop:non-neg}
Under \Cref{ass:A2,ass:A3} and
\Cref{ass:A5}--\Cref{ass:A8}, there is no surjection
\(M|_l\twoheadrightarrow\cO_l(-1)\) on an
\(H\)-degree-one line.
\end{proposition}

\begin{proof}
Suppose that such a quotient exists, and let
\(\tilde\gamma\subset\tilde T\) be the lifted quotient
curve. Crepancy and \Cref{ass:A2} give
\[
K_{\tilde T}\cdot\tilde\gamma=0,\qquad
\deg N_{\tilde\gamma/\tilde T}=-2.
\]
Since the normal bundle has rank three, Riemann--Roch
 give
\[
\dim_{[\tilde\gamma]}\operatorname{Hilb}(\tilde T)
\geq h^0(N_{\tilde\gamma/\tilde T})
     -h^1(N_{\tilde\gamma/\tilde T})=1.
\]
Choose a curve of distinct Hilbert points through
\([\tilde\gamma]\). After shrinking, its members are
smooth rational curves. Each member maps constantly
to the affine base. Its intersection with \(\mu^*D\)
remains one, whereas a curve over a point other than
\(P\) is disjoint from \(\mu^*D\). Thus every member
lies over \(P\).

The image under \(\beta\circ\mu\) has \(H\)-degree one.
Since \(H|_E\) is very ample, this image is a line,
and the map to it has degree one. The image in \(T\)
is therefore a quotient section of \(D=\PP_E(M)\)
with tautological degree \(-1\).

By \Cref{ass:A5,ass:A8}, a fixed
line supports at most one lifted quotient curve.
The lines must therefore vary. The corresponding
curves in \(T\) sweep a surface \(\Sigma\subset D\)
dominating a component of \(E\).

Every section in \(W\) vanishes on \(\Sigma\).
Its restriction to each quotient curve is a section
of \(\cO_{\PP^1}(-1)\), and hence zero. A general
anticanonical slice cuts \(\Sigma\) in a curve
\(\Gamma\) dominating a component of \(C\).
Thus \(\Gamma\) is a horizontal base curve for the
restriction of \(W_C\).

If \(M|_C=\cO(1)^{\oplus 2}\), then \(\dim W_C\geq3\),
contradicting \Cref{lem:horizontal-base}.
In every other case, on the component \(C_i\) dominated by \(\Gamma\)
we have \(M|_{C_i}=\cO\oplus\cO(d)\), with \(0\leq d\leq2\), and
\(\dim(W_C|_{C_i})\geq d+1\). By \Cref{lem:horizontal-base},
the subspace \(W_C|_{C_i}\) is the kernel of the quotient on sections
associated with \(\Gamma\). Thus \eqref{eq:W} makes
\(\bar r|_\Gamma\) nowhere zero.

On the other hand, \Cref{lem:fb} gives
\(f^*b\in\cI^2\). It defines a section of
\(\cL^2=\cO_T(-2D)\), whose restriction to every
negative quotient curve is zero. Its quadratic
initial section therefore vanishes on \(\Sigma\),
and hence on \(\Gamma\). On the slice, this section
is \(\bar r^2\). Thus
\(
\bar r^2|_\Gamma=0,
\) which yields a contradiction.
\end{proof}

\subsubsection{Reducedness of the fiber}

Once the conormal sheaf is globally generated, we
identify the scheme-theoretic fiber. In the locally
complete intersection case, the relevant statement
is \cite[Proposition 5.4]{AW1998}.

\begin{proposition}
\label{prop:aw-reducedness}
Under the contraction, locally complete intersection,
and klt blow-up hypotheses of
\Cref{lem:aw-conormal-lifting}, the conormal bundle
\(M\) is globally generated if and only if
\[
\fm_P\cO_X=\cI_{E/X}
\]
near \(E\). In this case, \(f^{-1}(P)=E\)
scheme-theoretically, and the blow-up of \(X\) along
\(E\) maps naturally to \(\operatorname{Bl}_P Y\).
\end{proposition}

The following additional hypothesis allows the same
argument beyond the locally complete intersection case.

\begin{assumption}[A9]
\assumptionlabel{ass:A9}{A9}
The blow-up \(T=\operatorname{Bl}_E X=\PP_X(\cI)\)
is klt, \(\cL\) is \(f\circ\beta\)-nef, and
\(\cL|_D\) is globally generated.
\end{assumption}

\begin{lemma}
\label{lem:fiber}
Under \Cref{ass:A2,ass:A9},
one has \(\fm_P\cO_X=\cI\) near \(E\).
\end{lemma}

\begin{proof}
Set \(p=f\circ\beta\). Since
\[
-2D-K_T=\beta^*H+3(-D),
\]
this divisor is \(p\)-nef and big. \(-D\)
is \(\beta\)-ample and \(H\) is \(f\)-ample, so
\(\beta^*H+\varepsilon(-D)\) is \(p\)-ample for
sufficiently small \(\varepsilon>0\).
Relative Kawamata--Viehweg vanishing gives
\[
R^1p_*\cO_T(-2D)=0.
\]
As \(Y\) is affine, the divisor sequence
\[
0\to\cO_T(-2D)\to\cO_T(-D)\to\cO_D(-D)\to0
\]
therefore gives a surjection
\[
H^0(T,\cL)\to H^0(D,\cL|_D).
\]

Choose generators of \(\cL|_D\) and lift them to
sections of \(\cL\). Through \(\cL\subset\cO_T\),
these are regular functions on \(T\).
Since \(p_*\cO_T=\cO_Y\), they are pullbacks of
functions \(a_j\) on \(Y\). Their vanishing on
\(D\) implies \(a_j\in\fm_P\).

For \(x\in E\), the equality \(T=\PP_X(\cI)\) gives
\[
\beta^{-1}(x)=\PP(\cI_x/\fm_x\cI_x),
\qquad
\cL|_{\beta^{-1}(x)}\cong\cO(1).
\]
The classes of \(f^*a_j\) have no common base point
on this projective space. They therefore span
\(\cI_x/\fm_x\cI_x\). Nakayama's lemma gives
\[
\cI_x=(f^*a_j)_j
\subseteq\fm_P\cO_{X,x}\subseteq\cI_x.
\]
Thus \(\fm_P\cO_X=\cI\) near \(E\), as required.
\end{proof}

\subsection{C2 case}

We begin with type C2 in \cite[Theorem 0.8]{Kachi1997}. Here \(E = A \cup_L B\) is the reduced union of two planes meeting along a line, with \(H\) of degree \(1\) on each plane. We will prove that the fiber is reduced and the target is smooth at \(P\).

The polarized surface is the reduced quadric in \(\PP^3\), so \Cref{ass:A1} is satisfied.
In the C2 case, we write \(\cO_E(t)=\cO_E(tH)\) and \(\cO_{E_n}(t)=\cO_X(tH)|_{E_n}\). In particular, \(\cO_{E_n}(-1)\cong\omega_X|_{E_n}\). The surface is Gorenstein with \(\omega_E\cong\cO_E(-2)\). A codimension-two Gorenstein subscheme of a regular scheme is a complete intersection \cite[Chapters 20--21]{Eisenbud1995}. Therefore
\(M=\cI/\cI^2\) is locally free of rank \(2\), \(\det M\cong\cO_E(1)\), and \(N\cong M^\vee\cong M(-1)\). Along the double line, local coordinates can be chosen so that
\[\cI = (z, xy), \quad \cI_{A/X} = (z, x), \cI_{B/X} = (z, y).\]
A regular-sequence ideal is of linear type, so \(T=\PP_X(\cI)\), \(D=\PP_E(M)\), and \(\cL|_D\cong\cO_D(1)\). In the local model above, the blow-up has equation \(xy=zu\), and thus it is normal and Gorenstein, with transverse threefold ordinary double points. The codimension-two blow-up formula gives \(K_T=\beta^*K_X+D\). Therefore \Cref{ass:A2} holds.

The polarization is very ample. Relative freeness and Bertini's theorem give a general \(D_h\) that is smooth near \(E\). The intersection \(C\) consists of two copies of \(\PP^1\) meeting at one point of \(L\). The embedding is regular, and the exact sequence
\[0 \to M \to \cO_{E_2} \to \cO_E \to 0, \]
with \(M\) locally free on \(E\), shows that \(\cO_{E_2}\) is also Cohen--Macaulay. A section avoiding both planes is therefore a nonzerodivisor on \(\cO_E\) and \(\cO_{E_2}\). These conditions are open, so \Cref{ass:A3} holds.

Since the image of the anticanonical system on either plane is \(H^0(\PP^2,\cO(1))\), every prescribed line, including \(L\), is contained in a suitable slice \(D_{h'}\) that contains neither plane. Such a section gives a plane conic \(C'=E\cap D_{h'}\), possibly a double line, and hence \(H^0(\cO_{C'})=\C\). The nonzerodivisor assertion follows as above. Finally, \(H^1(\cO_{E_2\cap D_{h'}})=0\) follows from \Cref{lem:aw-slice-vanishing} with \(F'=E_2\). Therefore \Cref{ass:A4} holds.

Now we check \Cref{ass:A5}.

\begin{lemma}
\label{lem:c2-a5}
    On every line \(l\) in either plane, including \(L\),
    \[M|_l\cong\cO \oplus \cO(1) \text{ or } \cO(-1) \oplus \cO(2).\]
\end{lemma}

\begin{proof}
    Choose the auxiliary slice \(C' = E \cap D_{h'}\) through \(l\). We have \(H^1(C', M|_{C'}) = 0\) and its quotient on \(l\) gives
    \(H^1(l, M|_l) = 0.\)
    Since \(M|_l\cong\cO(a) \oplus \cO(1-a)\), the conclusion then follows.
\end{proof}

Every \(H\)-degree-one line lies in one of the two planes. By \Cref{lem:c2-a5}, \(M|_l\) has either no negative quotient or a unique quotient to \(\cO_l(-1)\). This verifies \Cref{ass:A5}.

To pass from splitting types on lines to the bundle on a plane, we use the following result.
\begin{lemma}[{\cite[Lemma 2.7]{AW1998}}]
\label{lem:uniform-plane-bundles}
A rank-two vector bundle on \(\PP^2\) with the same splitting type on every line is either a direct sum of line bundles or a twist of \(T_{\PP^2}\). In particular, splitting \((-1,2)\) on every line forces the bundle to be \(\cO(-1)\oplus\cO(2)\), while splitting \((0,1)\) forces it to be \(\cO\oplus\cO(1)\) or \(T_{\PP^2}(-1)\).
\end{lemma}

We now rule out \((-1, 2)\) splitting for general lines.

\begin{lemma}
\label{lem:c2-pos}
    On a general line of each plane, \(M\) has splitting type \((0,1)\).
\end{lemma}

\begin{proof}
    Suppose the second splitting were generic on \(B\). The function \(h^0(l,M|_l(-2))\) is generically one, and then upper semicontinuity and the two-element splitting list force it to be one on every line. Then \Cref{lem:uniform-plane-bundles} gives \(M|_B\cong\cO_B(-1)\oplus\cO_B(2)\).

    Blow up \(A\), and call \(\hat{B}\) the strict transform of \(B\). It is still \(\PP^2\) and in the \(x\)-chart, write \(z = xa\), we have
    \(\cI \cO_{\operatorname{Bl}_AX} = x(a, y)\)
    and \(\hat{B}=(a,y)\). Dividing the ideal by the exceptional equation gives
    \[N^\vee_{\hat{B}/\operatorname{Bl}_AX}\cong M|_B(1)\cong\cO_B\oplus\cO_B(3).\]
    Thus the normal bundle is \(\cO_B\oplus\cO_B(-3)\), with \(h^0=1\) and \(h^1=0\). Hence \(\hat B\) deforms in a nonconstant smooth one-dimensional Hilbert family. Nearby surfaces are irreducible and are not contained in the exceptional divisor. They map to points of the affine base and must remain over \(P\), the isolated point with a two-dimensional fiber. Every such surface must therefore equal \(\hat B\), a contradiction. The same argument applies to \(A\).
\end{proof}

We now verify \Cref{ass:A6}. Choose a general slice \(C=C_A\cup C_B\). By \Cref{lem:c2-pos}, both component restrictions of \(M|_C\) are \(\cO\oplus\cO(1)\). Gluing sections with prescribed values at the node shows that the bundle is globally generated, hence nef. The slice is a crepant contraction of the smooth threefold \(D_h\) and its reduced fiber is locally complete intersection, and its blow-up has only an ordinary double point over the node. Therefore \Cref{lem:aw-conormal-lifting}, together with \cite[Proposition 5.6.3]{AW1998}, gives the surjectivity of \(\fm_{z_P}\to H^0(C,M|_C)\).

\begin{lemma}
    For \(W = \im(\fm_P \to H^0(E, M))\), one has \(\dim W \geq 3\).
\end{lemma}

\begin{proof}
    The normalization sequence identifies \(H^0(C,M|_C)\) with the kernel of the surjective difference-of-values map
    \[H^0(C_A,M|_{C_A})\oplus H^0(C_B,M|_{C_B})\longrightarrow M_p,\]
    where \(p=C_A\cap C_B\). Each component contributes three sections and \(\dim M_p=2\), so \(h^0(C,M|_C)=3+3-2=4\). By \Cref{ass:A6} and \eqref{eq:W}, the quotient \(H^0(C,M|_C)/W_C\) is generated by the single class of \(\bar r\), hence \(\dim W_C\geq3\). Finally, \(W\to W_C\) is surjective by definition, so \(\dim W\geq\dim W_C\geq3\).
\end{proof}

\begin{lemma}
    For \(k \geq 0\) and \(t \geq -1\),
    \(H^2(E, \Sym^k M(t)) = 0.\)
    Moreover, for every \(n \geq 1\),
    \(H^1(E_n, \cO_{E_n}) = 0\)
    and
    \(h^1(E, N) \leq 1.\)
\end{lemma}

\begin{proof}
    By Serre duality, the first vanishing is equivalent to
    \[H^0(E, (\Sym^k M)^{\vee}(-t - 2)) = 0.\]
    On a general line in either plane, \(M\) has splitting type \((0,1)\) by \Cref{lem:c2-pos}. So every summand of the displayed dual twist has strictly negative degree. And the global section therefore vanishes on a covering family of lines and so is zero.

    Regularity gives \(\cI^n/\cI^{n+1}\cong\Sym^n M\). The exact sequences
    \[0 \to \Sym^n M(t)  \to \cO_{E_{n + 1}}(t) \to \cO_{E_n}(t) \to 0\]
    and the \(H^2\)-vanishing make the transition map on \(H^1\) surjective for \(t \geq -1\).

    Let \(R = \cO_{Y, P}\) and \(\hat{R} = \varprojlim R/\fm_P^n\). We have
    \[\widehat{R^1f_* \cF}_P \cong \varprojlim H^1(E_n, \cF|_{E_n}).\]
    By \eqref{eq:aw}, \(R^1f_*\cO_X=0\) and \(R^1f_*\omega_X=\omega_Y\). Since \(\cO_{E_n}(-1)\cong\omega_X|_{E_n}\), and since the \(E_n\)-adic system is cofinal with the scheme-theoretic fiber neighborhoods, formal functions give
    \[\varprojlim H^1(E_n, \cO_{E_n}) = 0\]
    and
    \[\varprojlim H^1(E_n, \cO_{E_n}(-1)) \cong \widehat{(\omega_Y)_P}.\]
    Surjectivity of the transition maps gives \(H^1(E_n,\cO_{E_n})=0\). The first-neighborhood sequence and the acyclicity of \(\cO_E(-1)\) identify \(H^1(E_2,\cO_{E_2}(-1))\) with \(H^1(E,N)\), which is therefore a quotient of \(\widehat{\omega_{Y,P}}\). Since \(H^1(E,N)\) is annihilated by \(\fm_P\), it is a quotient of \(\C\). Hence \(h^1(E,N)\leq1\).
\end{proof}

We are now ready to verify \Cref{ass:A7}. Since \(N=M(-H)\), it suffices to prove \(H^0(E,N)=0\). Suppose, to the contrary, that \(s\in H^0(E,N)\) is nonzero. We claim that the section is nonzero on both \(A\) and \(B\). Otherwise, if it vanished on \(A\), for example, its restriction to \(B\) would vanish along \(L\) and define a section of \(M|_B(-2)\), which is zero by the generic splitting type \((0,1)\). Moreover, a nonzero section of \(N|_l=\cO\oplus\cO(-1)\) on a general line cannot vanish, so the zero locus of \(s\) contains no curve on either component.

Thus the zero scheme \(Z_s\) is empty or zero-dimensional. In the latter case, \(s\) is regular on the Cohen--Macaulay surface, so the Koszul sequence is
\[0 \to \cO_E \xrightarrow{s} N \to \cI_{Z_s/E}(-1) \to 0.\]
Because \(\cO_E(-1)\) is acyclic and \(H^{>0}(\cO_E)=0\), this gives
\[h^0(E, N) = 1, \quad h^1(E, N) = \operatorname{length}(Z_s) \leq 1.\]
If \(Z_s\) were empty, then \(h^0(N)-h^1(N)=1\), contradicting \Cref{lem:deform}. The bound \(h^1(N)\leq1\) therefore forces \(Z_s\) to be a single reduced point \(p\).

The point \(p\) lies in the smooth locus of \(E\). At a point of \(L\), the maximal ideal of the surface local ring has three minimal generators, whereas the ideal of a length-one regular zero of a rank-two section has two. This is impossible.

Twisting the Koszul sequence, we have

\begin{equation}
\label{eq:e1}
    0 \to \cO_E(1) \to M \to \cI_{p/E} \to 0.
\end{equation}
The evaluation \(H^0(\cO_E)\to\cO_p\) is an isomorphism. Hence \(\cI_{p/E}\) is acyclic, \(h^0(M)=h^0(\cO_E(1))=4\), and \(H^1(M)=0\).

Multiplication by the line subsheaf in \eqref{eq:e1} gives
\[0\to M(1)\to\Sym^2M\to\cI_{p/E}^2\to0.\]
The identification on the right follows from the local presentation of the regular ideal of \(p\).

By twisting \eqref{eq:e1}, we have \(H^1(M(1)) = H^2(M(1)) = 0\). Then the above exact sequence implies
\[H^1(\Sym^2M) = H^1(\cI_{p/E}^2) \cong H^0(\cO_E/\cI_{p/E}^2)/\C \cong \C^2.\]

Now consider the second-neighborhood sequence
\[0 \to \Sym^2 M \to \cO_{E_3} \to \cO_{E_2} \to 0,\]
Since \(H^1(\cO_{E_3})=0\), its boundary map is surjective. Moreover, \(H^0(\cO_{E_2})\cong\C\oplus H^0(M)\) as vector spaces. Constants lift from \(E_2\) to \(E_3\), so the boundary induces a surjection
\[\delta': H^0(E, M) \cong \C^4 \to H^1(E, \Sym^2 M) \cong \C^2.\]
Every base function is defined on \(X\), so its first-order class lifts from \(E_2\) to \(E_3\). Hence \(W\subset\ker\delta'\). But then
\[3 \leq \dim W \leq 4 - 2 = 2,\]
which is absurd. Therefore \Cref{ass:A7} holds.

Now we verify \Cref{ass:A8} by the following lemma.

\begin{lemma}
    The ordered blow-ups
    \(\tilde{T} = \operatorname{Bl}_{\hat{B}}(\operatorname{Bl}_{A}X)\)
    give a smooth small crepant resolution \(\mu:\tilde T\to T=\operatorname{Bl}_E X\). Here \(\hat B\) denotes the strict transform of \(B\) in \(\operatorname{Bl}_A X\). A degree-\(-1\) quotient curve over a line of \(E\) lifts uniquely to an embedded \(\PP^1\) of canonical degree \(0\).
\end{lemma}

\begin{proof}
    The centers \(A\) and \(\hat B\) are globally defined smooth surfaces. In the local model \(\cI=(z,xy)\), the first blow-up gives \(\cI\cO_{\operatorname{Bl}_A X}=\cO(-F_A)\cI_{\hat B}\), where \(F_A\) is its exceptional divisor, so the second blow-up makes the pullback of \(\cI\) invertible. The universal property of \(T=\operatorname{Bl}_E X\) therefore gives a global projective morphism \(\mu:\tilde T\to T\) and its local descriptions agree on overlaps by uniqueness in that universal property. Locally along the singular locus, \(T\) has equation \(xy=zu\), with one additional parameter along \(L\). The ordered blow-ups give the standard small resolution. Since \(T\) is Gorenstein and \(\mu\) has no exceptional divisor, \(\mu\) is crepant. The singular section has positive tautological degree, so a degree-\(-1\) quotient curve is not contained in it. Its lift away from finitely many points extends over \(\PP^1\) by properness and is unique because the lift is fixed on a dense open set. It is embedded because its composite with \(\mu\) is embedded. The canonical degree follows from crepancy.
\end{proof}

Finally, we obtain the following positivity statement.
\begin{proposition}
    The conormal bundle \(M\) is nef and globally generated.
\end{proposition}

\begin{proof}
    By \Cref{prop:non-neg}, the splitting type \((-1,2)\) is impossible on every line. Hence \Cref{lem:uniform-plane-bundles} gives
    \[M|_A, M|_B \in \{\cO\oplus \cO(1), T_{\PP^2}(-1)\}.\]
    Both bundles are globally generated, and their spaces of sections restrict surjectively to \(L\). Sections on the two components can therefore be chosen with matching restrictions and glued. Thus \(M\) is globally generated on \(E\), and in particular nef.
\end{proof}

For \Cref{ass:A9}, the klt property follows from the local ordinary-double-point model. The universal quotient makes \(\cL|_D\) globally generated and nef. Every curve in the special \(f\circ\beta\)-fiber lies in the support of \(D\), and every other contracted curve is \(\cL\)-trivial. Therefore \(\cL\) is \(f\circ\beta\)-nef.

We now prove our main theorem of the section.

\begin{theorem}
    \label{thm:c2}
    Suppose \(f: X^4 \to Y^3\) is an elementary \(K_X\)-negative contraction and
    \[E = A \cup_L B, A \cong B \cong \PP^2, \cO_A(H)\cong\cO_A(1),\quad\cO_B(H)\cong\cO_B(1),\]
    where \(L\) is a line on both planes. Then
    \[H^0(E, N_{E/X}) = 0, \fm_P \cO_X = \cI_{E/X} \text{ near } E, \widehat{\cO}_{Y,P} \cong \C\llbracket t_1, t_2, t_3 \rrbracket.\]
    Thus the scheme-theoretic fiber is reduced and \(Y\) is smooth at \(P\).
\end{theorem}

\begin{proof}
    By \Cref{lem:fiber}, \(\fm_P\cO_X=\cI\). This equality induces a morphism \(T\to\operatorname{Bl}_P Y\). For \(\xi=c_1(\cO_D(1))\), one has \(\xi^3=0\). The projective bundle formula then gives
    \[0 = \int_D \xi^3 = c_1(M)^2 - c_2(M) = 2 - c_2(M).\]
    The two possible plane restrictions have \(c_2=0\) or \(1\). Hence the restriction of \(M\) to both planes is \(T_{\PP^2}(-1)\). The gluing classification on the reducible quadric gives \(M=q^*T_{\PP^2}(-1)\), where \(q:E\to\PP^2\) identifies the two planes along \(L\). We have \(q_*\cO_E=\cO_{\PP^2}\oplus\cO_{\PP^2}(-1)\) and
    \(0 \to \cO_E(-1) \to \cO_E^{\oplus 3} \to M \to 0\). From Euler sequence, we have
    \[0 \to \cO(-1) \otimes \Sym^{k - 1} \C^3 \to \cO \otimes \Sym^k \C^3 \to \Sym^k T_{\PP^2}(-1) \to 0.\]
    Therefore
    \[H^1(E, \Sym^k M) = 0, H^0(E,\Sym^k M)\cong\Sym^k\C^3.\]
    Note that these identifications are compatible with multiplication. Hence
    \[H^0(E_n, \cO_{E_n}) \cong \C[t_1, t_2, t_3]/(t_1, t_2, t_3)^n.\]
    Formal functions identifies the inverse limit with \(\widehat{\cO}_{Y,P}\cong\C\llbracket t_1,t_2,t_3\rrbracket\). Since a Noetherian local ring is regular if and only if its completion is regular, \(Y\) is smooth at \(P\).
\end{proof}

\subsection{C3 case (\texorpdfstring{\(m = 3\)}{m = 3})}
By the reductions at the beginning of this section, it remains to consider \(E\cong S_3=\PP(1,1,3)\), the reduced cubic cone, with \(\cO_E(H)\cong\cO_E(3)\).

The surface \(E\) has a unique singular point \(v\), and its minimal resolution is
\[S\cong\bbF_3\cong \PP_{\PP^1}(\cO \oplus \cO(3)), \rho: S \to E.\]
The projection \(q:S\to\PP^1\) is the ruling and we write \(l\) for a fiber. The quotient \(\cO\oplus\cO(3)\twoheadrightarrow\cO\) defines the negative section \(\Delta\), which is contracted by \(\rho\). We have
\[\Delta^2 = -3, \Delta l = 1, l^2 = 0.\]
The pullback polarization is
\[H_S\coloneq\rho^*(H|_E)=\Delta+3l.\]

Let \(M=\cI/\cI^2\), let \(Q=M^{\vee\vee}\) be its reflexive hull, and put \(V=(\rho^*Q)^{\mathrm{tf}}\). The local calculation below shows that \(V\) is locally free.

At the vertex, the minimal Cohen presentation of the surface has embedding dimension \(4\), so its coordinates lift to coordinates in the completed smooth ambient ring. The ideal is
\[\cI = (ac - b^2, ad - bc, bd - c^2),\]
\[B = \begin{pmatrix}
    c & d \\
    -b & -c \\
    a & b
\end{pmatrix}\]
The minors of this matrix define the projective cone in \(\PP^4\), and the Hilbert--Burch theorem gives the resolution in \Cref{ass:A1}. Away from \(v\), \(E\) is smoothly embedded in codimension \(2\). Hence the Cohen--Macaulay condition also holds.

For \Cref{ass:A2}, consider the incidence equations
\[cx - by + az = 0, dx - cy + bz = 0\]
where \([x:y:z]\) are homogeneous coordinates on \(\PP^2\). For every such point, the two equations are linearly independent, so the incidence variety is the total space of a rank-two vector bundle over \(\PP^2\). It is irreducible and agrees with the graph away from \(E\). Thus
\(T = \operatorname{Bl}_E X = \PP_X(\cI).\)
The space \(T\) is smooth. The plane over \(v\) is not a divisor, so the discrepancy formula on the regular embedding locus extends to \(K_T=\beta^*K_X+D\). The symmetric-algebra description gives \(D=\PP_E(M)\) and \(\cL|_D\cong\cO_D(1)\). This proves \Cref{ass:A2}.

The parametrization \(a=u^3,b=u^2w,c=uw^2,d=w^3\) identifies the surface locally at \(v\) with a scalar \(\mu_3\)-quotient of \(\bbA^2\). Let \(M_1\) be the weight-one module. The conormal sheaf satisfies
\begin{equation}
    \label{eq:ex-c3}
0 \to M \to Q \to \C_v \to 0\end{equation}
where \(Q\cong M_1^{\oplus 2}\), and \(M\subset Q\) is generated by \((u,0),(w,u),(0,w)\). On the resolution, \(M_1\) pulls back modulo torsion to a line bundle of degree one on \(\Delta\). Thus \(V\) is locally free of rank two and \(V|_{\Delta}\cong\cO_{\Delta}(1)^{\oplus 2}\). The maximal minors \(u^2,uw,w^2\) have no common zero on \(\Delta\), so the natural map \(\rho^*M\to V\) is surjective at every point of \(\Delta\) by Nakayama's lemma. Away from \(\Delta\), the map is an isomorphism since \(\rho\) is an isomorphism there and \(M=Q\) away from \(v\).

We can now verify \Cref{ass:A3}. The cubic-cone polarization is very ample, and a general hyperplane avoids \(v\) and cuts a smooth rational curve. Relative freeness and Bertini make \(D_h\) smooth near \(E\). The embedding \(E\subset X\) is regular near this curve, and the conormal sheaf \(M\) is torsion-free. A hyperplane equation not vanishing identically on \(E\) is therefore a nonzerodivisor on \(\cO_E\), \(M\), and \(\cO_{E_2}\).

Adjunction on a general anticanonical slice, a smooth rational curve in a crepant threefold, gives conormal determinant degree \(2\). Writing \(c_1(V)=a\Delta+bl\), we obtain
\[c_1(V) H_S = b = 2, c_1(V)\Delta = -3a + b = 2.\]
Thus \(c_1(V)=2l\).

The same nonzerodivisor assertion holds for a hyperplane section \(C'\) through \(v\). Now \(H^0(\cO_E) = \C\) and \(H^i(\cO(-H)) = 0\) give \(H^0(\cO_{C'}) = \C\), even when \(C'\) is nonreduced. Then \Cref{lem:aw-slice-vanishing} gives \(H^1(\cO_{E_2 \cap D_{h'}}) = 0\).

Since we could choose \(C'\) to contain a prescribed ruling by base-point freeness of the subsystem, we verify all statements in \Cref{ass:A4}.

The following lemma bounds the splitting types of the conormal restrictions.
\begin{lemma}
\label{lem:c3-split}
    Every ruling has \(V|_l\cong\cO^{\oplus 2}\) or \(\cO(-1)\oplus\cO(1)\). On a general section \(C\in\abs{H_S}\) avoiding \(\Delta\), the splitting type is one of
    \[(1,1), (0,2), (-1,3).\]
\end{lemma}

\begin{proof}
    By \Cref{ass:A4}, we have \(H^1(C',M|_{C'})=0\). For a slice through \(l\), the natural map from the conormal restriction onto the pushforward of \(V|_l\) is surjective, so \(H^1(l,V|_l)=0\).

    Identifying the ruling with its image in \(E\), the surjection \(M|_{\rho(l)} \twoheadrightarrow V|_l\) kills the torsion. The map induces a surjection between rank \(2\) vector bundles on \(\PP^1\), and hence an isomorphism.

    In particular, every quotient of \(V|_l\) is a quotient of the actual conormal restriction. Its determinant has degree \(0\), proving the ruling list. On a general slice, determinant degree \(2\) and the same vanishing give the second list.
\end{proof}

Every \(H\)-degree-one line on the cubic cone is a ruling. Its torsion-free conormal restriction is therefore \(V|_l\), with splitting type \((0,0)\) or \((-1,1)\). This verifies \Cref{ass:A5}.

To establish \Cref{ass:A6}, it remains to exclude the \((-1,3)\) case.

\begin{lemma}
\label{lem:c3-push}
    \(\rho_*V^{\vee}\cong N\) and \(R^1\rho_*V^{\vee}=0\).
\end{lemma}

\begin{proof}
    The surjection \(\rho^*M\twoheadrightarrow V\) has torsion kernel, so \(V^\vee\cong\cHom_S(\rho^*M,\cO_S)\). Since \(E\) is normal and \(\rho\) is proper birational, \(\rho_*\cO_S=\cO_E\). The pullback--pushforward adjunction for sheaf Hom gives
    \[\rho_*V^\vee\cong\cHom_E(M,\rho_*\cO_S)=N.\]

    For \(R^1\rho_*V^\vee\) at \(v\), consider the formal neighborhoods of \(\Delta\). Since \(\cO_S(-\Delta)|_{\Delta}=\cO_{\Delta}(3)\) and \(V^\vee|_{\Delta}=\cO(-1)^{\oplus 2}\), the exact sequences of the thickenings give
    \[H^1(\Delta^{(n)}, V^{\vee}|_{\Delta^{(n)}}) = 0.\]
   Thus
    \[\widehat{(R^1\rho_*V^\vee)_v}\cong\varprojlim_n H^1(\Delta^{(n)},V^\vee|_{\Delta^{(n)}})=0.\]
\end{proof}

To exclude the negative generic splitting type, we will descend the maximal line subsheaf from the incidence variety to the resolution \(S\). The following criterion provides this descent.
\begin{proposition}[{\cite[Lemma 2.1 and Proposition 2.2]{CF2014}}]
\label{prop:cf-descent}
Let \(p:T\to S\) be a dominant morphism of smooth connected varieties, with all nonempty fibers of dimension \(\dim T-\dim S\) and with irreducible general fibers. Let \(V\) be a vector bundle on \(S\), let \(L\subset p^*V\) be saturated, and put \(Q=p^*V/L\). If the second fundamental form
\[L\longrightarrow\Omega_{T/S}\otimes Q\]
vanishes generically, then \(L=p^*L_0\) for a saturated subsheaf \(L_0\subset V\). When \(S\) is a smooth surface and \(L_0\) has rank one, \(L_0\) is a line bundle.
\end{proposition}

We apply this criterion to obtain the following description of \(V\).

\begin{lemma}
\label{lem:c3-exact}
    Suppose the generic splitting type on sections of \(H_S\) is \((-1,3)\). Then there is an exact sequence
    \[0 \to \cO_S(\Delta + 3l) \to V \to \cO_S(-\Delta - l) \to 0.\]
\end{lemma}

\begin{proof}
   The complement \(S \backslash \Delta\) is the total space \(\cO_{\PP^1}(3)\). So the section avoiding \(\Delta\) is an affine translation of a base-point-free vector space \(U \subset H^0(\PP^1, \cO(3))\). The evaluation sequence is
   \[0 \to M_U \to U \otimes \cO \to \cO(3) \to 0.\]
   Since \(U\hookrightarrow H^0(\cO(3))\) has no kernel, \(M_U\) splits as \(\bigoplus_j\cO_{\PP^1}(-a_j)\), with \(1\leq a_j\leq3\) and \(\sum a_j=3\).

   On the incidence of these graphs
   \[\mathcal U = \{(x, C): x \in C\} \subset (S \backslash \Delta) \times U,\]
   the degree \(3\) summand is a unique maximal line subsheaf \(L\) over a dense open set of graph parameters. The tangent space along fibers of the evaluation map to \(S \backslash \Delta\) restricts on a graph to \(M_U\). As \(C \in U\) varies, the \((L_C)_x \subset V_x\) moves with it. Its second fundamental form has components
   \[\cO(3) \to \cO(-1 + a_j),\]
   which are zero. The line subsheaf is therefore constant along the connected open evaluation fibers. So it descends to a line subsheaf of \(V\) by \Cref{prop:cf-descent}.

   Saturating this line subsheaf in \(V\) gives a line bundle. Since its degree on a graph is \(3\), \cite[Section 1]{Hartshorne1980} gives an exact sequence
   \[0 \to \cO_S(a\Delta + 3l) \to V \to \cI_{Z/S} \otimes \cO_S(-a\Delta - l) \to 0,\]
   where \(Z\) is zero-dimensional. Since saturation ensures that the map does not vanish generically on \(\Delta\), and \(V|_{\Delta}\cong\cO(1)^{\oplus 2}\), we have \(3-3a\leq1\). On a general ruling the line subsheaf has degree \(a\leq1\) by \Cref{lem:c3-split}. Hence \(a=1\).

   If \(Z\) met a ruling \(l\), the restricted degree-one line subsheaf would vanish there, and its saturation on that ruling would have degree at least two, contradicting \Cref{lem:c3-split}. Thus \(Z=\emptyset\).
\end{proof}

\begin{proposition}
    The generic splitting type \((-1,3)\) is impossible.
\end{proposition}

\begin{proof}
     Dualizing the exact sequence in the preceding lemma gives
    \[0 \to \cO_S(\Delta + l) \to V^{\vee} \to \cO_S(-\Delta - 3l) \to 0.\]
    The last term has ruling degree \(-1\) and is acyclic. Moreover,
    \[Rq_* \cO_S(\Delta + l) \cong \cO(1) \oplus \cO(-2).\]
    Therefore by \Cref{lem:c3-push}, \((h^0,h^1,h^2)(E,N)=(2,1,0)\), which contradicts \Cref{lem:deform}.
\end{proof}

The remaining conormal bundles on a general slice are \(\cO(1)^{\oplus 2}\) and \(\cO\oplus\cO(2)\). The slice is a smooth rational curve in the smooth crepant threefold \(D_h\), and its blow-up is smooth. The required surjectivity follows from \Cref{lem:aw-conormal-lifting} and \cite[Proposition 5.6.2]{AW1998}. Thus \Cref{ass:A6} holds.

On a general section of \(H_S\), the two remaining splitting types become \((-2,-2)\) or \((-3,-1)\) after twisting by \(-H_S\). Such sections cover \(S\), so \(H^0(S,V(-H_S))=0\). The injection \(M\hookrightarrow\rho_*V\) then gives \(H^0(E,M(-H))=0\). This proves \Cref{ass:A7}. Together with \Cref{ass:A3}, \Cref{lem:fb} gives \(f^*b\in\cI^2\).

Since \(T\) is smooth, \Cref{ass:A8} is satisfied by taking \(\tilde T=T\) and \(\mu=\id\).

Now we have our positivity result.
\begin{proposition}
    Every ruling has trivial restriction, and \(V\cong q^* \cO_{\PP^1}(1)^{\oplus 2}\).
\end{proposition}

\begin{proof}
    By \Cref{prop:non-neg}, a splitting type \((-1,1)\) is impossible. Hence \(V\) is trivial on every ruling. The evaluation map \(q^*q_*V\to V\) is an isomorphism on every fiber and therefore globally. The restriction \(V|_{\Delta}\cong\cO(1)^{\oplus 2}\) gives \(q_*V\cong\cO(1)^{\oplus 2}\).
\end{proof}

\begin{theorem}
    \label{thm:c3-exceptional}
    In the C3(\(m=3\)) case, with special fiber
    \(E \cong \PP(1,1,3),\)
    the sheaf \(\cO_E\) is exceptional.
\end{theorem}

\begin{proof}
    Pushing forward the bundle \(V\), we have
    \[Q\cong\cO_E(1)^{\oplus 2},\quad N\cong\cO_E(-1)^{\oplus 2}.\]
    Weighted-projective cohomology gives \(R\Gamma(N)=0\). The exceptional property now follows from \Cref{lem:c-exc}.
\end{proof}

The same calculation also yields reducedness and smoothness.
\begin{lemma}
    The conormal sheaf \(M\) is generated by three global sections.
\end{lemma}

\begin{proof}
    The three local generators \((u,0),(w,u),(0,w)\) globalize. The four sections of \(Q\cong\cO_E(1)^{\oplus 2}\) map onto \(Q/M\cong\C_v\). After a constant change of basis, the kernel has the displayed form. These sections have rank two away from \(v\) and minimally generate \(M\) at \(v\). Hence \(M\) is generated by three global sections.
\end{proof}

We now verify \Cref{ass:A9}. The tautological quotient line bundle on \(D = \PP_E(M)\) is globally generated and nef. Every curve of the special \(f\circ\beta\)-fiber lies in \(D\), and off the fiber \(\cL\) is trivial on contracted curves. Thus \(\cL\) is \(f\circ\beta\)-nef. The blow-up is smooth and therefore klt.

\begin{theorem}
    \label{thm:c3-smooth}
    The C3(\(m=3\)) fiber is scheme-theoretically reduced and \(Y\) is smooth at \(P\).
\end{theorem}

\begin{proof}
    By \Cref{lem:fiber}, we have \(\fm_P \cO_X = \cI\), with three lifted base functions \(s_1, s_2, s_3\) generating \(\cI\) near \(E\).

    Let \(\cK_Y\) be the Koszul complex of \(s_1,s_2,s_3\) on \(Y\), and put \(\cK_X=Lf^*\cK_Y\). The complex \(\cK_X\) has precisely
    \[\mathcal H^0(\cK_X)\cong\cO_E,\qquad \mathcal H^{-1}(\cK_X)\cong\cG=\omega_E(H)\]
    as its nonzero cohomology sheaves. Since \(R\Gamma(\cG)=0\), we have \(Rf_*\cK_X\cong\C_P\).

    The projection formula and \(Rf_*\cO_X\cong\cO_Y\) give
    \(Rf_*\cK_X\cong\cK_Y.\) In particular, \(\cO_Y/(s_1,s_2,s_3)\cong\C_P\). Thus the three elements generate \(\fm_P\), and the local ring has dimension three and embedding dimension at most three, hence is regular.
\end{proof}

\subsection{C5 case}

We next consider type C5 in \cite[Theorem 0.8]{Kachi1997}, where \(E=A\cup_LB\), \(A\cong\PP^2\), \(B\cong S_2=\PP(1,1,2)\), and the common line passes through the vertex \(v\) of \(B\). The polarizations are \(\cO_A(H)\cong\cO_A(1)\) and \(\cO_B(H)\cong\cO_B(2)\).

As in the C3 case, we first consider the equations locally at \(v\). In suitable coordinates,
\[\cI=(be,ce,b^2-ac),\qquad \cI_{A/X}=(b,c),\qquad \cI_{B/X}=(e,b^2-ac),\qquad \cI_{L/X}=(b,c,e).\]
The two relations are the columns of
\[\begin{pmatrix}
    c & b \\
    -b & -a \\
    0 & -e \\
\end{pmatrix}\]
This is the minimal Cohen presentation, and the embedding dimension is again \(4\). The matrix above is a Hilbert--Burch matrix for this height-two ideal, giving the resolution in \Cref{ass:A1}. The embedding is regular away from \(v\) and is Cohen--Macaulay of codimension two at \(v\). Thus \Cref{ass:A1} holds.

We now verify \Cref{ass:A2}. The symmetric-algebra incidence equations in projective coordinates \([u_0:u_1:u_2]\) are
\[cu_0-bu_1=0,\qquad bu_0-au_1-eu_2=0.\]
They form a regular sequence in the polynomial algebra. The charts \(u_0=1\) and \(u_1=1\) are smooth. In the \(u_2=1\) chart, the incidence variety is irreducible and has only transverse ordinary double points. Thus it agrees with the graph on a dense open set, so \(T=\operatorname{Bl}_E X=\PP_X(\cI)\) and is normal and Gorenstein. Moreover, \(D=\PP_E(M)\) and \(\cL|_D\cong\cO_D(1)\). The discrepancy formula holds on the regular embedding locus and the fiber over \(v\) is a plane of codimension two in \(T\), so the formula extends to \(K_T=\beta^*K_X+D\). This verifies \Cref{ass:A2}.

For \(i\in\{A,B\}\), put \(M_i=M|_i\) and \(Q_A=M_A^{\vee\vee}\).
The completed stalks at \(v\) are
\[
\widehat{(M_A)_v}\cong\widehat{\cO}_{A,v}\oplus\widehat{(\cI_{v/A})_v},
\qquad
\widehat{(M_B)_v}\cong\widehat{\cO}_{B,v}\oplus M_1,
\]
where \(M_1\) is the weight-one reflexive module of the scalar \(\mu_2\)-quotient.
The local free generators on \(A\) and \(B\) are the classes of \(be\) and
\(b^2-ac\), respectively. Since \(M_A\) is locally free away from \(v\),
the completed stalk calculation gives \(Q_A/M_A\cong\C_v\).

Resolve both components
\[\rho_A: S_A=\operatorname{Bl}_v A\cong\bbF_1 \to A,\quad
 \rho_B: S_B\cong\bbF_2 \to B,\]
and \(V_i=(\rho_i^*M_i)^{\mathrm{tf}}\). For \(i\in\{A,B\}\), let \(\Delta_i,l_i\) be the negative section and a ruling, and let \(q_i:S_i\to\PP^1\). We use \(\Delta_i\) for these exceptional sections so that \(C_A,C_B\) remain available for the components of an anticanonical slice. The local modules show that \(V_i\) is locally free and
\[V_i|_{\Delta_i}\cong\cO \oplus \cO(1),\quad
\rho_A^*(H|_A) = \Delta_A + l_A,\quad
\rho_B^*(H|_B) = \Delta_B + 2l_B.\]

\begin{lemma}
\label{lem:c5-nzd}
    A hyperplane equation containing neither surface component is a nonzerodivisor on \(\cO_E\), \(M\), and \(\cO_{E_2}\). The module \(M\) has no finite-length subsheaf supported at \(v\).
\end{lemma}

\begin{proof}
    We first show that \(\cO_{E_2}\) has no embedded point at \(v\). We take the local setup, with
    \[R=\widehat{\cO}_{X,v}\cong\C\llbracket a,b,c,e\rrbracket,\qquad I=\cI_{E/X,v}R.\]
    The kernel of the graded surjection \(\operatorname{Sym}_R(I)\to\bigoplus_{n\geq0}I^n\) is supported on the zero section by the equality of projectivizations proved above. The symmetric algebra is cut out by two equations forming a regular sequence in \(R[u_0,u_1,u_2]\). Its associated primes therefore have height two in the polynomial ring, whereas the ideal \((u_0,u_1,u_2)\) has height three. Thus the symmetric algebra has no nonzero submodule supported on the zero section, and the graded surjection is an isomorphism.
    The degree-two strand of the Koszul resolution of the two incidence relations gives a resolution
    \[0 \to R \to R^6 \to R^6 \to \Sym^2_R(I)\cong I^2 \to 0\]
    of length two.

    Since \(R\) is a regular local ring, \(\operatorname{pd}_R(I^2) \leq 2.\)
    From 
    \[0\to I^2\to R\to R/I^2\to0,\]
    we obtain \(\operatorname{pd}_R(R/I^2) \leq 3.\) The Auslander--Buchsbaum theorem then gives \(\depth_R(R/I^2) \geq 1,\) and therefore \(\depth_v \cO_{E_2} \geq 1.\) Thus the maximal ideal corresponding to \(v\) is not an associated prime of \(\cO_{E_2}\).

    Therefore if \(h\) vanishes identically on neither \(A\) nor \(B\), multiplication by \(h\) on \(\cO_{E_2}\) is injective. It follows that \(h\) is also a nonzerodivisor on \(M\) and \(\cO_E\).

    The depth calculation also excludes a finite-length subsheaf of \(M\) supported at \(v\). Such a submodule would contribute to
    \(H^0_{\fm_v}(\cO_{E_2, v})\)
    and would force depth zero, contradicting the preceding calculation.
\end{proof}

Since a general hyperplane avoids \(v\), Bertini's theorem and \Cref{lem:c5-nzd} show that \Cref{ass:A3} holds.

Adjunction on the resulting two-component crepant curve gives conormal determinant degree one on each component. Write \(c_1(V_i)=a\Delta_i+bl_i\). Its degree on a general anticanonical slice is \(b=1\), while \(V_i|_{\Delta_i}\cong\cO\oplus\cO(1)\) gives \(-ma+b=1\), where \(m=1\) or \(2\). Hence
\(c_1(V_i)=l_i.\)

Fix an \(H\)-degree-one line in either component. A relative section \(D_{h'}\) can be chosen to contain it without containing either surface. For \(C'=E\cap D_{h'}\), \Cref{lem:c5-nzd} shows that \(h'\) is a nonzerodivisor. Since \(H^0(\cO_E)=\C\) and \(\cO_E(-H)\) is acyclic, the hyperplane sequence gives \(H^0(\cO_{C'})=\C\). The first-neighborhood vanishing follows from \Cref{lem:aw-slice-vanishing}. Thus \Cref{ass:A4} is verified.

Now we consider the bounds on splitting types.
\begin{lemma}
    On every ruling, \(V_i|_{l_i}\) is isomorphic to \(\cO^{\oplus 2}\) or \(\cO(-1)\oplus\cO(1)\). On a general anticanonical slice component, the restriction is isomorphic to \(\cO\oplus\cO(1)\) or \(\cO(-1)\oplus\cO(2)\). The latter list also holds for every plane line avoiding \(v\).
\end{lemma}

\begin{proof}
    We have \(H^1(C',M|_{C'})=0\) for a slice containing a prescribed line. Passing to the quotient supported on that line gives \(H^1(l_i,V_i|_{l_i})=0\). The determinant degrees then give the stated possibilities.
\end{proof}

The surjection \(\rho_i^*M_i\twoheadrightarrow V_i\), restricted to a ruling, identifies the actual torsion-free conormal restriction, as in the C3 case:
\begin{equation}
\label{eq:c5-vm}
    (M|_{\rho_i(l_i)})^{\mathrm{tf}} \cong V_i|_{l_i}.
\end{equation}

We now verify \Cref{ass:A5}. By the preceding lemma, the splitting type is \((0,0)\) or \((-1,1)\) for lines through \(v\), and \((0,1)\) or \((-1,2)\) for a plane line avoiding \(v\). In each negative case there is a unique quotient to \(\cO(-1)\), up to scalar, and in the other cases there is none.

\begin{lemma}
\label{lem:c5-splitting}
    If the negative generic splitting type occurs on component \(A\) or \(B\), then
    \[V_i\cong\cO(\Delta_i + 2 l_i) \oplus \cO(- \Delta_i - l_i).\]
    For the plane \(A\), it implies that
    \[M_A = \cI_{v/A}(2) \oplus \cO_A(-1)\]
    and
    \[Q_A\cong\cO_A(2) \oplus \cO_A(-1).\]
\end{lemma}

\begin{proof}
    Suppose the generic splitting type is \((-1,2)\) on \(S_i=\bbF_m\), with \(m=1\) or \(2\). As in \Cref{lem:c3-exact}, use the incidence of graphs in the total space of \(\cO_{\PP^1}(m)\). Its evaluation kernel splits as \(\bigoplus_j\cO(-a_j)\), with \(1\leq a_j\leq m\). The second fundamental forms of the maximal degree-two line subsheaf are
    \[\cO(2) \to \cO(-1 + a_j)\]
    and hence vanish. By \Cref{prop:cf-descent}, the line subsheaf descends and saturates to
    \[0\to\cO(a\Delta_i+2l_i)\to V_i\to\cI_{Z/S_i}\otimes\cO(-a\Delta_i-l_i)\to0,\]
    where \(Z\) is zero-dimensional. Restriction at the generic point of \(\Delta_i\) gives \(2-ma\leq1\), while the ruling bound gives \(a\leq1\). Thus \(a=1\). A point of \(Z\) on a ruling would increase the saturated ruling degree above one, contradicting the preliminary bounds on splitting types. Hence \(Z=\emptyset\).

    The extension group is \(H^1(\bbF_m,\cO(2\Delta_i+3l_i))\), which vanishes for \(m=1,2\). Thus \(V_i\) is a direct sum. On \(\bbF_1\), we have \(\Delta_A+2l_A=2\rho_A^*h-\Delta_A\) and \(-\Delta_A-l_A=-\rho_A^*h\). Pushing forward then gives the final two identities.
\end{proof}

We now do some local computations. Let \(\pi: \hat{X} = \operatorname{Bl}_A X \to X\) have exceptional divisor \(F\). The strict transform \(\hat{B}\) is \(\bbF_2\), which may be identified with the resolution \(S_B\). We retain \(\Delta_B\) for its negative section and write \(\gamma_0\) for the strict transform of the common line \(L\). In the \(b\)-chart, write \(c = bt\). Then
\begin{equation}
    \label{eq:c5:b-chart}
    \cI \cO_{\hat{X}} = b(e, b - at),\quad \hat{B} = (e, b - at).
\end{equation}

On \(\hat B\), the exceptional equation is \(b=at\), so \(\Delta_B=(a=0)\) and \(\gamma_0=(t=0)\). In the \(c\)-chart, set \(b=cu\). Then \(\cI\cO=c(e,cu^2-a)\). These two charts describe the standard resolution of the quadric cone. They also prove
\begin{equation}
    \label{eq:c5-blow-up}
    \begin{aligned}
    \cI\cO_{\hat X}=\cO_{\hat X}(-F)\cI_{\hat B/\hat X}, F|_{\hat{B}} = \Delta_B + \gamma_0, \\
    \pi^* H|_{\hat{B}} = \Delta_B + 2 l_B, N^\vee_{\hat{B}/\hat{X}}\cong V_B(\Delta_B + l_B).
    \end{aligned}
\end{equation}

The following proposition uses isolation to exclude the negative generic splitting type.

\begin{proposition}
    The generic splitting type on a quadric slice cannot be \((-1, 2)\), independent of the plane splitting.
\end{proposition}

\begin{proof}
    If it were, then \Cref{lem:c5-splitting} and \eqref{eq:c5-blow-up} would give
    \[N_{\hat{B}/\hat{X}}\cong\cO_{\bbF_2} \oplus \cO_{\bbF_2}(- 2\Delta_B - 3l_B).\]
    Since the relative canonical bundle is \(\omega_{q_B}\cong\cO_{S_B}(-2\Delta_B-2l_B)\), relative duality gives
    \[Rq_{B*}\cO_{S_B}(-2\Delta_B-3l_B)\cong\cO_{\PP^1}(-1)[-1].\]
    This complex is acyclic. Thus the normal bundle has \(h^0=1\) and \(h^1=0\), so the smoothly embedded \(\hat B\) has a nonconstant one-dimensional Hilbert family. Nearby members are proper smooth irreducible surfaces not contained in \(F\), and therefore remain over \(P\). The special fiber has support \(F\cup\hat B\), so every member must equal \(\hat B\), a contradiction.
\end{proof}

It remains to exclude the negative generic splitting type on the plane. We assume this splitting type and derive a contradiction.

\begin{lemma}
    There is a reduced point \(w \in L \backslash \{v\}\) such that
    \[0 \to \cO_A \to N_{A/X} \to \cI_{w/A}(-2) \to 0.\]
    Put \(U = N^\vee_{A/X}(1)\). Then we have exact sequences
    \[0 \to \cO_A(3) \to U \to \cI_{w/A}(1) \to 0\]
    \[0 \to \cO_A(-1) \to \cO_A(3) \oplus \cO_A^{\oplus 2} \to U \to 0.\]
\end{lemma}

\begin{proof}
    The conormal inclusion extends reflexively to
    \[0 \to Q_A \to N^\vee_{A/X} \to \cO_L(-1) \to 0.\]
    At \(v\), the image of \(M_A\) in the conormal bundle generated by \(b,c\) is generated by \((eb, ec, ac)\). Its reflexive extension is generated by \(eb, c\), with cokernel a line module over \(e = 0\). Globally this cokernel is \(\cO_L(-1)\).

    Adjunction gives \(\det N_{A/X}=\cO_A(-2)\). With \(h=c_1(\cO_A(1))\), we have \(\ch(\cO_L(-1))=h-\frac32h^2\). Combining this with \(Q_A\cong\cO(2)\oplus\cO(-1)\) gives
    \(c_2(N_{A/X}) = 1.\)

    The line subsheaf \(\cO(2)\) defines a section of \(N^\vee_{A/X}(-2)\cong N_{A/X}\). It does not vanish at \(v\), because the defect \(Q_A/M_A\cong\C_v\) lies in the direction of that line subsheaf. Away from \(L\), the conormal inclusion is an isomorphism, so the section is nonzero. Its zero scheme is therefore zero-dimensional, contained in \(L\), and has length \(c_2=1\). It is a reduced point \(w\neq v\).

    For rank two, \(N^\vee_{A/X}(1)\cong N_{A/X}(3)\), so twisting gives the first exact sequence. Resolve \(\cI_{w/A}(-2)\) by \[0\to\cO(-4)\to\cO(-3)^2\to\cI_{w/A}(-2)\to0.\] Since \(H^1(\cO_A(3))=0\), the extension lifts to \[0\to\cO_A(-4)\to\cO_A\oplus\cO_A(-3)^{\oplus 2}\to N_{A/X}\to0.\] Twisting finally gives the second exact sequence.
\end{proof}

We identify \(F = \PP_A(N^\vee_{A/X}) = \PP_A(U)\), with \(p:F \to A\). Its tautological divisor class is
\[\xi=c_1(\cO_{\PP_A(U)}(1))=(\pi^*H-F)|_F.\]
The line subsheaf \(\cO_A(3) \to U\) cuts out the surface \(S \cong \operatorname{Bl}_w A \cong \bbF_1\). Let \(\Delta_S\) and \(l_S\) be the negative section and ruling on \(S\) with \(q: S \to \PP^1\). Then
\[p^* h|_S = \Delta_S + l_S, \xi|_S = l_S, F|_S = \Delta_S.\]
The normal sequence gives
\begin{equation}
\label{eq:c5-normal-seq}
    0 \to \cO_S(-3\Delta_S - 2l_S) \to N_{S/\hat{X}} \to \cO_S(\Delta_S) \to 0.
\end{equation}

The curve \(\gamma_0 \subset \hat{B}\) also lies in \(S\). It maps isomorphically to \(L\) and  \(\xi \gamma_0 = 0\). We denote
\[v_0 = \gamma_0 \cap \Delta_B, \quad w_0 = \gamma_0 \cap \Delta_S.\]
These are distinct points over \(v\) and \(w\), respectively.

We summarize the geometry in the following diagram.
\[\begin{tikzcd}
	S & {\hat{X}} & {\hat{B}} \\
	A & X & B
	\arrow[hook, from=1-1, to=1-2]
	\arrow["\rho"', from=1-1, to=2-1]
	\arrow["\pi", from=1-2, to=2-2]
	\arrow[hook', from=1-3, to=1-2]
	\arrow["{\rho_B}", from=1-3, to=2-3]
	\arrow[hook, from=2-1, to=2-2]
	\arrow[hook', from=2-3, to=2-2]
\end{tikzcd}\]
The two surfaces \(S\) and \(\hat{B}\) share \(\gamma_0 \cong L\).

\begin{lemma}
\label{lem:c5-ruling}
    The surfaces \(S\) and \(\hat B\) have independent normal directions along their common ruling \(\gamma_0\). In particular,
    \[N_{\gamma_0/\hat{X}}\cong\cO^{\oplus 2} \oplus \cO(-2)\]
    and
    \[N_{S/\hat{X}}|_{\gamma_0}\cong\cO \oplus \cO(-2).\]
    The image of \(H^0(N_{S/\hat{X}}|_{\gamma_0}) \to H^0(\cO_{\gamma_0}(F))\) is the one-dimensional subspace of sections vanishing at \(v_0\).
\end{lemma}

\begin{proof}
    Each ruling normal \(N_{\gamma_0/S}\) and \(N_{\gamma_0/\hat B}\) is \(\cO\). The first direction maps to zero in \(\cO_{\gamma_0}(F)\), since \(S\subset F\), while the second has a simple zero at \(v_0\). Thus they are independent away from \(v_0\). At \(v_0\), use the coordinates \((a,b,t,e)\) in \eqref{eq:c5:b-chart}. The curve is \((b,t,e)=(0,0,0)\), and the direction in \(\hat B=(e,b-at)\) transverse to it is \(\partial_t+a\partial_b\), which becomes \(\partial_t\) at \(a=0\) and is killed by \(d\pi\). On the other hand, the morphism \(S\cong\operatorname{Bl}_w A\to A\) is an isomorphism near \(v\), since \(w\neq v\). A direction in \(S\) transverse to \(\gamma_0\) therefore has nonzero \(\partial_e\)-component, as \(L=(e=0)\) on \(A\). These directions remain independent modulo \(T_{v_0}\gamma_0\).

    These two directions form a subbundle \(\cO^{\oplus 2} \subset N_{\gamma_0/\hat{X}}\). Now
    \[K_{\hat{X}} \gamma_0 = (\pi^* K_X + F) \gamma_0 = 0.\]
    Adjunction gives normal degree \(-2\), so the quotient is \(\cO(-2)\). The extension splits because \(H^1(\cO(2))=0\). The final assertion follows by quotienting by \(N_{\gamma_0/S}\).
\end{proof}

\begin{lemma}
\label{lem:c5-connect}
    Applying \(Rq_*\) to \eqref{eq:c5-normal-seq} and taking cohomology sheaves gives a connecting map
    \[\delta: \cO \oplus \cO(-1) \to \cO(-1) \oplus \cO, \delta = \begin{pmatrix}
        0 & a \\
        a & b \\
    \end{pmatrix}\]
    where \(a \in \C\) and \(b \in H^0(\PP^1, \cO(1))\).
\end{lemma}

\begin{proof}
    Since \(\omega_q\cong\cO_S(-2\Delta_S-l_S)\), the projective bundle formula and relative duality give
    \[\mathcal E\coloneq q_*\cO_S(\Delta_S)\cong\cO \oplus \cO(-1)\]
    \[R^1q_*\cO_S(-3\Delta_S-2l_S)\cong (\mathcal E)^\vee \otimes \cO(-1)\]
    and the other images vanish. The extension class belongs to
    \[H^1(S,\cO_S(-4\Delta_S-2l_S))\cong H^0(\PP^1, \cO(-1) \oplus \cO \oplus \cO(1)).\]
    Relative Serre duality identifies its cup-product map with a symmetric form \(\mathcal E\to\mathcal E^\vee\otimes\cO(-1)\). Symmetry follows from \(\Sym^2\mathcal E\cong q_*\cO_S(2\Delta_S)\), which gives the displayed matrix form.
\end{proof}

\begin{proposition}
    The generic splitting type on the plane is nonnegative.
\end{proposition}

\begin{proof}
    Let \(t_0 \in \PP^1\) correspond to \(\gamma_0\). Restricting \eqref{eq:c5-normal-seq} gives
    \[0 \to \cO_{\gamma_0}(-3) \to \cO_{\gamma_0} \oplus \cO_{\gamma_0}(-2) \to \cO_{\gamma_0}(1) \to 0.\]
    The fiber connecting map \(H^0(\cO(1))\to H^1(\cO(-3))\) has a one-dimensional kernel. Its determinant in \Cref{lem:c5-connect} is \(-a^2\), so \(a=0\).

    Therefore the first summand \(\cO\subset\mathcal E\) lies in the kernel globally. The canonical section of \(\cO_S(\Delta_S)\) lifts to \(N_{S/\hat X}\). On \(\gamma_0\), its image in \(H^0(\cO_{\gamma_0}(F))\) is nonzero and vanishes at \(w_0\). By \Cref{lem:c5-ruling}, every nonzero image vanishes at \(v_0\). Hence \(w_0=v_0\), contradicting \(w\neq v\).
\end{proof}

Therefore only the splitting type \((0,1)\) occurs on a general slice of either component. Each component bundle is globally generated, and sections can be glued by matching their values at the node. Thus \(M|_C\) is globally generated and nef.

The slice contracts this reduced two-component curve crepantly in a smooth threefold. Its blow-up has an ordinary double point over the node and is therefore klt. \Cref{lem:aw-conormal-lifting} and \cite[Proposition 5.6.3]{AW1998} give the surjection \(\fm_{z_P}\to H^0(C,M|_C)\). This verifies \Cref{ass:A6}.

We now verify \Cref{ass:A7}. On \(B\), restriction to a moving general hyperplane has splitting type \((-2,-1)\), so any section of \(M(-H)\) restricts to zero on \(B\). Its restriction to \(A\) therefore vanishes on \(L\). Mapping it into \(Q_A\) gives a section of \(Q_A(-2)\), whose splitting type on a general line is also \((-2,-1)\). Hence this section is zero. The kernel of \(M\to M_A\oplus M_B\) is supported at \(v\), and \Cref{lem:c5-nzd} excludes such a finite-length subsheaf. Thus \(H^0(E,M(-H))=0\), and \Cref{lem:fb} gives \(f^*b\in\cI^2\).

We next verify \Cref{ass:A8} for \(\tilde T=\operatorname{Bl}_{\hat B}(\operatorname{Bl}_A X)\). Here \(A\) is smooth and the globally defined strict transform \(\hat B\cong\bbF_2\) is smooth. By \eqref{eq:c5-blow-up}, the second blow-up makes the pullback of \(\cI\) invertible. Thus the universal property of \(T=\operatorname{Bl}_E X\) gives a global projective morphism \(\mu:\tilde T\to T\), whose local descriptions agree on overlaps. The local incidence charts show that \(\tilde T\) is smooth and that \(\mu\) has no exceptional divisor. If \(F_A,F_B\) are the exceptional divisors over \(X\), then
\[\mu^* D = F_A + F_B,\quad K_{\tilde{T}} = \mu^* K_T.\]

The singular section of \(T\) has tautological degree two, so a degree-\(-1\) quotient curve is not that section and meets the singular locus only finitely. Its lift extends by properness and is unique by separatedness. It is embedded because its composite into \(T\) is embedded. This completes the verification.

\begin{proposition}
\label{prop:c5-ruling}
    Every ruling restriction of \(V_A\), \(V_B\) is trivial. Consequently
    \[V_i\cong q_i^*(\cO \oplus \cO(1)),\quad
    M_A\cong\cO_A \oplus \cI_{v/A}(1),\quad
    M_B\cong\cO_B \oplus \cO_B(1),\quad
    Q_A\cong\cO_A \oplus \cO_A(1).\]
\end{proposition}

\begin{proof}
    \Cref{prop:non-neg} excludes every ruling splitting type \((-1,1)\). The bundles therefore descend along their ruling maps, and restriction to \(\Delta_i\) shows that each descended bundle is \(\cO\oplus\cO(1)\).

    Finally, \(\rho_{i*}V_i\cong M_i\). This follows from the local description. In particular, we have
    \[\rho_{A*}\cO_{S_A}(l_A)\cong\cI_{v/A}(1),\quad
    \rho_{B*}\cO_{S_B}(l_B)\cong\cO_B(1).\]
    These identities give all the asserted formulas.
\end{proof}

\begin{lemma}
\label{lem:c5-van}
    Let \(M_L^{\mathrm{tf}}=(M|_L)^{\mathrm{tf}}\). There is an exact sequence
    \[0 \to N \to Q_A^{\vee} \oplus M_B^{\vee} \to (M_L^{\mathrm{tf}})^\vee \to 0.\]
    The restrictions of \(V_A\) and \(V_B\) to the strict transforms of \(L\) identify with \(M_L^{\mathrm{tf}}\), and \(R\Gamma(E,N)=0\).
\end{lemma}

\begin{proof}
    We apply \(\cHom_E(M, -)\) to the normalization sequence
    \[0 \to \cO_E \to \cO_A \oplus \cO_B \to \cO_L \to 0.\]
    At the vertex \(v\), surjectivity follows from the independent free generators \(be\) on \(A\) and \(b^2-ac\) on \(B\).

    Equation \eqref{eq:c5-vm} identifies \(M_L^{\mathrm{tf}}\) with the restrictions of the resolved bundles to the strict transforms of \(L\), which are rulings. Thus \Cref{prop:c5-ruling} gives \(M_L^{\mathrm{tf}}\cong\cO_L^{\oplus 2}\).

    Thus the exact sequence becomes
    \[0 \to N \to (\cO_A \oplus \cO_A(-1)) \oplus (\cO_B \oplus \cO_B(-1)) \to \cO_L^{\oplus 2} \to 0.\]
    The negative summands \(\cO_A(-1)\) and \(\cO_B(-1)\) are acyclic, so the middle term has only \(H^0=\C^2\). Its map to \(H^0(\cO_L^{\oplus 2})=\C^2\) is an isomorphism because the two generators are independent at \(v\). The cohomology sequence gives \(R\Gamma(E,N)=0\).
\end{proof}

\begin{theorem}
    \label{thm:c5}
    In the C5 case, with
    \[E = A \cup_L B,\quad A \cong \PP^2, \quad B \cong \PP(1,1,2),\quad
    \cO_A(H)\cong\cO_A(1),\quad\cO_B(H)\cong\cO_B(2),\]
    the sheaf \(\cO_E\) is exceptional.
\end{theorem}

\begin{proof}
    This is a direct consequence of \Cref{lem:c5-van,lem:c-exc}.
\end{proof}

\subsection{\texorpdfstring{\(\mathrm{C}'\)}{C'} case}

In case \(\mathrm{C}'\), both the length and the analytic relative Picard number are one by \cite[Theorems 0.5--0.7]{Kachi1997}. The polarization on each plane is \(\cO(1)\) by \cite[Proposition 10.6]{Kachi1997}.

By \cite[Lemma 4.8]{AW1998}, a general anticanonical section of \(E\) is connected. The positive-dimensional components of the pairwise intersections are lines, and \cite[Lemmas 4.7.1 and 4.10.3 and Proposition 4.11]{AW1998} exclude a common line and a cycle of three distinct lines. We may therefore write \(E=A_1\cup B\cup A_2\), where \(A_i\) meets \(B\) along a line \(L_i\). Put \(v=L_1\cap L_2\).

\begin{lemma}
\label{lem:cprime-local}
    The scheme-theoretic intersections are \(A_i\cap B=L_i\) and \(A_1\cap A_2=\{v\}\), and the end planes meet transversely at \(v\). In suitable formal coordinates \((a,b,c,d)\) at \(v\),
    \[
    \cI_{A_1/X}=(a,b),\qquad \cI_{B/X}=(a,d),\qquad
    \cI_{A_2/X}=(c,d),\qquad \cI=(ad,bd,ac).
    \]
\end{lemma}

\begin{proof}
    Put \(Z_i=A_i\cap B\). The relative anticanonical system restricts to the full system of lines on each plane. Through any prescribed point of \(Z_i\), we can therefore choose a slice \(D_h\) containing no plane and with \(h|_{L_i}\neq 0\).

    Let \(S\) be the scheme-theoretic union of \(A_i\cap D_h\) and \(B\cap D_h\). By \Cref{lem:aw-slice-vanishing}, \(H^1(\cO_{(A_i\cup B)\cap D_h})=0\). The kernel of the restriction to \(\cO_S\) has support of dimension at most one, so \(H^1(\cO_S)=0\) as well. The exact sequence for the union of the two lines gives
    \[\operatorname{length}(Z_i\cap D_h)=1.\]
    Since \(L_i\cap D_h\) already has length one, \(Z_i\) is supported on \(L_i\). At every point of \(L_i\), the local ring of \(Z_i\) modulo the slice equation is \(\C\). Thus the local ring is regular of dimension one and \(Z_i=L_i\) scheme-theoretically.

    Any point of \(A_1\cap A_2\) other than \(v\) would lie outside \(B\). A slice through such a point and avoiding \(v\) would cut out, after reduction, a cycle of three lines meeting at three distinct points. This curve has nonzero \(H^1(\cO)\), contradicting \Cref{lem:aw-slice-vanishing}.

    Suppose the end planes fail to meet transversely at \(v\). The sum of their tangent planes then has dimension at most three and contains \(T_vB=T_vL_1+T_vL_2\). Choose a slice through \(v\) whose differential is nonzero on each \(T_vL_i\). The three branches of the reduced slice have tangent directions spanning a space of dimension at most two. In the normalization, the three constant terms coincide and the linear coefficients satisfy at least one relation. The normalization quotient thus has length at least three at \(v\). The three constants on the normalized lines map to the quotient with the diagonal in the kernel, so the normalization sequence gives nonzero \(H^1(\cO)\). This again contradicts the vanishing.

    By transversality, the two pairs of generators for the ideals of the end planes form a system of parameters. Since each \(L_i\) is Cartier on \(B\), we may choose one generator in each pair that vanishes on \(B\). The other two generators restrict to parameters on \(B\). This gives the stated coordinates and ideals.
\end{proof}

The local model in \Cref{lem:cprime-local} gives the exact sequence
\begin{equation}
\label{eq:three-planes-gluing}
0 \to \cO_E \to \cO_{A_1}\oplus\cO_B\oplus\cO_{A_2}
\to \cO_{L_1}\oplus\cO_{L_2} \to 0,
\end{equation}
where the last map takes differences along the two lines. The polarized gluing identifies \(E\) with the corresponding union of coordinate planes in \(\PP^4\). Its homogeneous ideal is generated by the maximal minors of
\[
\begin{pmatrix}
 b & c\\
 -a & 0\\
 0 & -d
\end{pmatrix}.
\]
By the Hilbert--Burch theorem, \(E\) has the resolution in \Cref{ass:A1} and is Cohen--Macaulay at \(v\). Along \(L_i\setminus\{v\}\), the two planes have ideals \((z,x)\) and \((z,y)\), so \(\cI=(z,xy)\). Thus the embedding is regular away from \(v\), and \Cref{ass:A1} holds.

For \Cref{ass:A2}, the incidence equations for \(\PP_X(\cI)\) at \(v\), in projective coordinates \([u_0:u_1:u_2]\) corresponding to \((ad,bd,ac)\), are
\[
bu_0-au_1=0,\qquad cu_0-du_2=0.
\]
The chart \(u_0=1\) is smooth. On the chart \(u_1=1\), eliminating \(a\) gives the equation \(cu_0=du_2\). The chart \(u_2=1\) is analogous. These charts and the local models away from \(v\) show that \(\PP_X(\cI)\) is normal and Gorenstein of pure dimension four. The preimage of \(E\) has dimension three, with line fibers away from \(v\) and a plane over \(v\). Hence \(\PP_X(\cI)\) is the closure of the graph over \(X\setminus E\), and \(T=\PP_X(\cI)\).

The symmetric-algebra description gives \(D=\PP_E(M)\) and \(\cL|_D\cong\cO_D(1)\). The discrepancy formula for the regular embedding extends across the plane over \(v\), which has codimension two in \(T\). Thus \(K_T=\beta^*K_X+D\), and \Cref{ass:A2} holds.

\begin{lemma}
\label{lem:cprime-nzd}
    A hyperplane equation avoiding the components of \(E\) is a nonzerodivisor on \(\cO_E\), \(M\), and \(\cO_{E_2}\). The sheaf \(M\) has no finite-length subsheaf supported at \(v\).
\end{lemma}

\begin{proof}
    Put \(R=\widehat{\cO}_{X,v}\) and \(I=\cI R\). The two incidence relations form a regular sequence in the polynomial algebra over \(R\). Since \(T=\PP_X(\cI)\), the associated-prime argument of \Cref{lem:c5-nzd} shows that \(I\) is of linear type. The degree-two strand of the Koszul resolution gives \(\operatorname{pd}_R(I^2)\leq 2\), hence \(\operatorname{pd}_R(R/I^2)\leq 3\). By the Auslander--Buchsbaum formula, \(\depth_v\cO_{E_2}\geq 1\).

    Away from \(v\), the embedding is regular and \(\cO_{E_2}\) is Cohen--Macaulay. The depth bound excludes an embedded point at \(v\). Thus any hyperplane equation avoiding the components is a nonzerodivisor on \(\cO_{E_2}\) and on its subsheaf \(M\). It is also a nonzerodivisor on \(\cO_E\) since \(E\) is reduced. The depth bound excludes a finite-length subsheaf of \(M\) supported at \(v\).
\end{proof}

A general anticanonical slice avoids \(v\) and cuts \(E\) in a chain of three lines with one node on each \(L_i\). The polarization is very ample by \Cref{ass:A1}, and the embedding is regular along the chain. Bertini's theorem gives smoothness of \(D_h\) near \(E\), and \Cref{lem:cprime-nzd} gives the nonzerodivisor assertions. This verifies \Cref{ass:A3}.

Fix a line \(l\subset E\). Vanishing on \(l\) imposes two linear conditions on the relative anticanonical system, and vanishing on any plane imposes three. We may therefore choose \(D_{h'}\) containing \(l\) and containing no plane. By \Cref{lem:cprime-nzd}, \(h'\) is a nonzerodivisor on \(\cO_{E_2}\). For \(C'=E\cap D_{h'}\), the hyperplane sequence and \eqref{eq:vanishing} give \(H^0(C',\cO_{C'})=\C\), and \Cref{lem:aw-slice-vanishing} gives \(H^1(\cO_{E_2\cap D_{h'}})=0\). The same argument applies to a general slice. Thus \Cref{ass:A4} holds.

Write \(M_i=M|_{A_i}\), \(Q_i=M_i^{\vee\vee}\), and \(Q_B=(M|_B)^{\mathrm{tf}}\). Locally at \(v\), the presentation in \Cref{lem:cprime-local} gives \(M_i\cong\cO_{A_i}\oplus\cI_{v/A_i}\). The free summand is generated by \([bd]\) on \(A_1\) and by \([ac]\) on \(A_2\). On \(B\), the classes \([bd]\) and \([ac]\) generate a free summand, and \([ad]\) generates the torsion of length one. In particular, \(Q_B\) is locally free.

Let \(\rho_i:S_i=\operatorname{Bl}_v A_i\to A_i\) be the blow-up and put \(V_i=(\rho_i^*M_i)^{\mathrm{tf}}\). Write \(q_i:S_i\to\PP^1\) for the ruling of \(S_i\cong\bbF_1\), \(\Delta_i\) for its negative section, and \(l_i\) for a fiber. The pullback of the ideal summand modulo torsion is \(\cO_{S_i}(-\Delta_i)\), whose pushforward is \(\cI_{v/A_i}\). Hence \(V_i\) is locally free and
\[
V_i|_{\Delta_i}\cong\cO\oplus\cO(1),\qquad
\rho_i^*(H|_{A_i})=\Delta_i+l_i,\qquad
M_i\cong\rho_{i*}V_i.
\]
Adjunction on a general crepant slice gives \(\det(M|_C)\cong\omega_C^{-1}\), of degree one on each end component and zero on the middle component. The degrees on the slice component and on \(\Delta_i\) give \(c_1(V_i)=l_i\). The degree on the middle component gives \(c_1(Q_B)=0\).

\begin{lemma}
\label{lem:cprime-split}
    Let \(l\subset E\) be an \(H\)-degree-one line.
    \begin{enumerate}[label=(\roman*)]
        \item If \(l\) lies in \(B\) or is a line through \(v\) in an end plane, then \((M|_l)^{\mathrm{tf}}\) has splitting type \((0,0)\) or \((-1,1)\).
        \item If \(l\) lies in an end plane and avoids \(v\), then \((M|_l)^{\mathrm{tf}}\) has splitting type \((0,1)\) or \((-1,2)\).
    \end{enumerate}
\end{lemma}

\begin{proof}
    Choose a slice \(C'\) through \(l\) as in \Cref{ass:A4}. The exact sequence for the first infinitesimal neighborhood and the quotient onto \((M|_l)^{\mathrm{tf}}\) give
    \[
    H^1(C',M|_{C'})=0,\qquad H^1(l,(M|_l)^{\mathrm{tf}})=0.
    \]

    For a line in \(B\), the torsion-free restriction is \(Q_B|_l\). For a line through \(v\) in \(A_i\), the surjection \(\rho_i^*M_i\twoheadrightarrow V_i\) identifies the torsion-free restriction with \(V_i\) restricted to the corresponding fiber of \(q_i\). Both have degree zero. For a line in \(A_i\) avoiding \(v\), the degree is one. These degrees and the \(H^1\)-vanishing give the stated splitting types.
\end{proof}

In each negative splitting in \Cref{lem:cprime-split}, the quotient of degree \(-1\) is unique up to scalar. This verifies \Cref{ass:A5}.

\begin{proposition}
\label{prop:cprime-generic-split}
    The generic splitting type on \(B\) is \((0,0)\), and on each end plane it is \((0,1)\).
\end{proposition}

\begin{proof}
    On \(B\), the conormal inclusion gives the exact sequence
    \begin{equation}
    \label{eq:three-planes-middle-normal}
    0 \to Q_B \to N_{B/X}^{\vee} \to
    \cO_{L_1}(-1)\oplus\cO_{L_2}(-1) \to 0.
    \end{equation}
    The last map is induced by restriction to the conormal bundles of \(L_i\subset A_i\). At \(v\), the classes \([bd]\) and \([ac]\) map to \(b[d]\) and \(c[a]\), respectively. Away from \(v\), the sequence follows from the local model along \(L_i\).

    For each \(j\), the resolution
    \[
    0 \to \cO_B(-2) \to \cO_B(-1) \to \cO_{L_j}(-1) \to 0
    \]
    has first map given by multiplication by an equation of \(L_j\).
    On \(L_i\) this map is zero for \(j=i\) and injective for \(j\neq i\),
    since in the latter case the equation vanishes only at \(v\).
    Thus the only nonzero \(\operatorname{Tor}_1\) term is
    \(\operatorname{Tor}_1^{\cO_B}(\cO_{L_i}(-1),\cO_{L_i})
    \cong\cO_{L_i}(-2)\). Since \(N_{B/X}^{\vee}\) is locally free,
    the Tor sequence obtained by tensoring \eqref{eq:three-planes-middle-normal} with \(\cO_{L_i}\)
    over \(\cO_B\) identifies this term with the kernel of
    \(Q_B|_{L_i}\to N_{B/X}^{\vee}|_{L_i}\).

    Away from \(v\), the model \(\cI=(z,xy)\), with \(B\) given by
    \((z,x)\), identifies the kernel with the summand generated by
    \([xy]\). Setting \(b=0\) in the stalk at \(v\) on \(L_1\)
    kills the image of \([bd]\), while \([ac]\) maps to \(c[a]\).
    Since \(c\) is a parameter on \(L_1\) at \(v\), the kernel is
    the summand generated by \([bd]\). Similarly, on \(L_2\),
    setting \(c=0\) gives the summand generated by \([ac]\).
    Hence \(\cO_{L_i}(-2)\) is a saturated line subbundle of \(Q_B|_{L_i}\).

    Suppose the generic splitting type on \(B\) is \((-1,1)\). By upper semicontinuity and \Cref{lem:cprime-split}, this type occurs on every line. Then \Cref{lem:uniform-plane-bundles} gives \(Q_B\cong\cO_B(-1)\oplus\cO_B(1)\). On each \(L_i\), the summand \(\cO_{L_i}(1)\) and the subbundle \(\cO_{L_i}(-2)\) have distinct generic fibers, since a saturated line subbundle is determined by its generic fiber. Thus the map \(\cO_B(1)\to N_{B/X}^{\vee}\) vanishes at only finitely many points.

    On a general line \(l\subset B\), the map gives a line subbundle of degree one. By adjunction, \(N_{B/X}^{\vee}|_l\) has determinant degree two, so \(N_{B/X}^{\vee}|_l\cong\cO_l(1)^{\oplus 2}\). The normal sequence of \(l\subset B\subset X\) gives \(h^0(N_{l/X})=2\) and \(h^1(N_{l/X})=0\). Thus the two-dimensional family of lines in \(B\) fills the smooth Hilbert scheme germ at \([l]\). This contradicts \cite[Theorem 9.1]{Kachi1997}, which gives a divisor swept out by deformations of a length-one line when the analytic relative Picard number is one.

    Suppose the generic splitting type on \(A_i\) is \((-1,2)\). The incidence argument of \Cref{lem:c5-splitting} applies to \(V_i\) with the same bounds on ruling fibers and the same restriction to \(\Delta_i\). It gives \(Q_i\cong\cO_{A_i}(2)\oplus\cO_{A_i}(-1)\).

    Blow up \(B\). Since \(L_i\) is Cartier on \(A_i\), the strict transform of \(A_i\) is isomorphic to \(A_i\). Away from \(v\), dividing the pullback ideal by the exceptional factor in the model \((z,xy)\) identifies its conormal bundle with \(Q_i(1)\). The identification extends across \(v\) because both sheaves are locally free. Thus the normal bundle is \(\cO\oplus\cO(-3)\), with \(h^0=1\) and \(h^1=0\).

    The strict transform would therefore move in a one-dimensional Hilbert family. Nearby members are proper irreducible surfaces, so each maps to a point of the affine base. They meet the complement of the exceptional divisor, and isolation of the two-dimensional fiber forces them to lie over \(P\). The two strict transforms are the only surface components over \(P\) meeting that complement. This contradicts the existence of a nonconstant family.
\end{proof}

By \Cref{prop:cprime-generic-split}, the conormal bundle on a general slice has the splitting types in \Cref{ass:A6}(iii). Sections on the components glue by matching their values at the two nodes, so \(M|_C\) is globally generated. The slice is crepant, and its blow-up has an ordinary double point over each node. By \Cref{lem:aw-conormal-lifting}, the map \(\fm_{z_P}\to H^0(C,M|_C)\) is surjective. This verifies \Cref{ass:A6}.

We now verify \Cref{ass:A7}. By the generic splitting on \(B\), any section of \(M(-H)\) has zero image in \(Q_B(-1)\). Its image in \(Q_i(-1)\) vanishes along \(L_i\), since the conormal restrictions agree away from \(v\). The image therefore defines a section of \(Q_i(-2)\), which vanishes by the generic splitting on \(A_i\). Since \(M\) is locally free away from \(v\), \eqref{eq:three-planes-gluing} gives vanishing there. By \Cref{lem:cprime-nzd}, a section supported at \(v\) is zero as well. Thus \(H^0(E,M(-H))=0\).

\begin{lemma}
\label{lem:cprime-resolution}
    Blow up \(A_1\), then the strict transform of \(A_2\), and finally the strict transform of \(B\). These blow-ups give a projective small crepant resolution \(\mu:\tilde T\to T\). Every quotient \(M|_l\twoheadrightarrow\cO_l(-1)\) on an \(H\)-degree-one line determines a curve in \(T\) that lifts uniquely to a curve mapped isomorphically onto it.
\end{lemma}

\begin{proof}
    After blowing up \(A_1\), divide the pullback of \(\cI\) by the exceptional factor. On the chart \(a=bs\), the resulting ideal is \((d,sc)\), and the strict transforms of \(B\) and \(A_2\) have ideals \((s,d)\) and \((c,d)\), respectively. On the chart \(b=at\), the resulting ideal is \((d,c)\).

    The strict transform of \(A_2\) is the blow-up of \(A_2\) at \(v\). The strict transform of \(B\) is isomorphic to \(B\) because \(L_1\) is Cartier on \(B\). The two transforms meet along a Cartier line on this copy of \(B\). The two remaining blow-ups are those of the smooth surfaces in the local model \((z,xy)\). They make the pullback of \(\cI\) invertible. By the universal property of the blow-up, the resulting smooth variety admits a projective birational morphism \(\mu:\tilde T\to T\).

    Each divisor of \(\tilde T\) over \(E\) dominates one of the three planes and maps isomorphically to the corresponding divisor of \(T\) over a general point of that plane. Hence \(\mu\) is small and crepant, and is an isomorphism over the smooth locus of \(T\).

    The incidence charts show that the singular locus of \(T\) consists of a section over each \(L_i\). The corresponding quotient of \(Q_B|_{L_i}\) kills the \(\cO_{L_i}(-2)\) subbundle from \eqref{eq:three-planes-middle-normal}, which corresponds to the generator \(xy\) in the local model. Thus the quotient has degree two, and a degree-\(-1\) quotient curve meets the singular locus in at most finitely many points. Its lift over the smooth locus extends by properness and is unique by separatedness. The lift is a closed immersion because it is a section over the original curve.
\end{proof}

This verifies \Cref{ass:A8}. We can now determine the conormal restrictions.

\begin{proposition}
\label{prop:cprime-conormal}
    Every ruling restriction of \(V_i\) is trivial, and
    \[
    \begin{aligned}
        Q_B&\cong\cO_B^{\oplus 2},&
        V_i&\cong q_i^*(\cO\oplus\cO(1)),\\
        M_i&\cong\cO_{A_i}\oplus\cI_{v/A_i}(1),&
        Q_i&\cong\cO_{A_i}\oplus\cO_{A_i}(1).
    \end{aligned}
    \]
\end{proposition}

\begin{proof}
    By \Cref{prop:non-neg,lem:cprime-split}, the restriction of \(V_i\) to each ruling fiber and of \(Q_B\) to each line is trivial. Thus \(Q_B\cong\cO_B^{\oplus 2}\) by \Cref{lem:uniform-plane-bundles}. Each \(V_i\) descends along \(q_i\), and restriction to \(\Delta_i\) identifies the descended bundle with \(\cO\oplus\cO(1)\). Since \(\rho_{i*}\cO_{S_i}(l_i)\cong\cI_{v/A_i}(1)\), pushing forward gives \(M_i\), and taking the reflexive hull gives \(Q_i\).
\end{proof}

\begin{lemma}
\label{lem:cprime-van}
    The normal sheaf fits into an exact sequence
    \[
    0 \to N \to Q_1^\vee\oplus Q_B^\vee\oplus Q_2^\vee
    \to (Q_B|_{L_1})^\vee\oplus(Q_B|_{L_2})^\vee \to 0.
    \]
    In particular, \(R\Gamma(E,N)=0\).
\end{lemma}

\begin{proof}
    Apply \(\cHom_E(M,-)\) to \eqref{eq:three-planes-gluing}, using \((M|_{L_i})^{\mathrm{tf}}\cong Q_B|_{L_i}\). Since \(M\) is locally free away from \(v\), the right-hand map is surjective there. At \(v\), the classes \([bd]\) and \([ac]\) form a basis of each torsion-free conormal fiber on the two lines. The maps from the end planes give the directions dual to \([bd]\) on \(L_1\) and to \([ac]\) on \(L_2\). The map from \(B\) gives the same pair of coordinates on both lines. Together these images span the four-dimensional target fiber. This proves surjectivity at \(v\).

    By \Cref{prop:cprime-conormal}, the middle term is a direct sum of four trivial line bundles and two copies of \(\cO(-1)\) on the respective planes. The right-hand term is trivial of rank two on each line. Both terms have vanishing higher cohomology and \(H^0\cong\C^4\).

    The projection \(M_i\to\cO_{A_i}\) is surjective. Its restriction to the local ideal summand has image in the maximal ideal at \(v\), since every map \(\cI_{v/A_i}\to\cO_{A_i}\) is locally multiplication by a regular function. Thus the projection is nonzero on the local free generator at \(v\). Evaluating global sections at \(v\) now gives the same spanning directions as above. Hence the map on global sections is an isomorphism, and the cohomology sequence gives \(R\Gamma(E,N)=0\).
\end{proof}

\begin{theorem}
\label{thm:three-planes}
    If the reduced surface fiber is of type \(\mathrm{C}'\), then \(\cO_E\) is exceptional.
\end{theorem}

\begin{proof}
    This follows from \Cref{lem:cprime-van,lem:c-exc} and the verification of \Cref{ass:A1}.
\end{proof}

\section{The fourfold case}
\label{sec:fourfold}
In this section, we treat the fourfold case using the classification
of extremal contractions and the results of
\Cref{sec:refinement}. We construct nonzero proper admissible
subcategories for the birational contractions, including those with
singular target, and then complete the argument for contractions
of fiber type.

\begin{lemma}
    Let \(f:X\to Y\) be an elementary \(K_X\)-negative contraction of fiber type from a smooth fourfold, with \(\dim Y\leq2\). Then \(Y\) is smooth.
\end{lemma}

\begin{proof}
    The smoothness of \(Y\) follows from \cite[Proposition 1.2, Corollary 1.4]{AW1997}.
\end{proof}

It remains to consider contractions onto a threefold. The equidimensional case has smooth target and is again covered by \Cref{lem:smooth}. The non-equidimensional case, with isolated surface fibers, is covered by the analysis in \Cref{sec:refinement}, based on \cite[Theorems 0.5--0.8]{Kachi1997}. We now treat the birational contractions. We label such a contraction by
\[(\dim\operatorname{Exc}(f),\dim f(\operatorname{Exc}(f))).\]

\subsection{The \texorpdfstring{\((2, 0)\)}{(2, 0)} case}
By \cite[Theorem 1.1]{Kawamata1989}, if \(f\) is a small extremal contraction from a smooth fourfold, then its exceptional locus is a disjoint union of copies of \(\PP^2\), and for every component \(E\),
\[N_{E/X} \cong \cO_{\PP^2}(-1) \oplus \cO_{\PP^2}(-1).\] In this case \(\cO_E\) is exceptional. The Koszul calculation in \Cref{prop:koszul-background} gives
\[\cExt_X^q(\cO_E, \cO_E) \cong \wedge^q N_{E/X}, \quad q = 0,1,2\]
and these sheaves vanish for \(q>2\). The local-to-global spectral sequence is
\[E^{p,q}_2 = H^p(E, \wedge^q N_{E/X}) \Rightarrow \Ext^{p + q}_X(\cO_E, \cO_E)\]
The only nonzero cohomology term is \(H^0(E,\cO_E)=\C\), so \(R\Hom_X(\cO_E,\cO_E)=\C\). By \cite[Theorem 2.3]{AW1997}, we have
\[\dim\operatorname{Exc}(f)+\dim F\geq4+\ell-1\geq4,\]
where \(F\) is a nontrivial fiber component and \(\ell\geq1\) is the minimum anticanonical degree of a rational curve in \(F\). For types \((2,1)\) and \((1,0)\), the left-hand side is respectively \(3\) and \(2\) over a general point of the image of the exceptional locus, so these types are excluded.

\subsection{The \texorpdfstring{\((3, 0)\)}{(3, 0)} case}

This follows from \Cref{lem:contra-point}.

\subsection{The \texorpdfstring{\((3, 1)\)}{(3, 1)} case}

By \cite[Main Theorem and Theorem 1.1]{Takagi1999}, the image of the exceptional divisor is a smooth curve \(C\), and the induced morphism \(g:E\to C\) is a \(\PP^2\)-bundle or a quadric bundle without multiple fibers. We have a commutative diagram

\[\begin{tikzcd}
	E & X \\
	C & Y
	\arrow["i", from=1-1, to=1-2]
	\arrow["g"', from=1-1, to=2-1]
	\arrow["f", from=1-2, to=2-2]
	\arrow["j"', from=2-1, to=2-2]
\end{tikzcd}\]

\begin{lemma}
    \(Ri_*Lg^*\Db(C)\) is an admissible subcategory of \(\Db(X)\).
\end{lemma}

\begin{proof}
    By \cite[Main Theorem]{Takagi1999}, \(g:E\to C\) is either a \(\PP^2\)-bundle or a quadric bundle. In both cases, we have
    \(Rg_*\cO_E\cong\cO_C.\)
    Also, we have
    \(Rg_*\cO_E(E) = 0\),
    which follows from the adjunction formula
    \(K_F = (K_X + E)|_F\)
    and the polarized fibers listed in \cite[Theorem 1.1]{Takagi1999}. On a plane fiber, \(\cO_E(E)\) restricts to \(\cO_{\PP^2}(-2)\) or \(\cO_{\PP^2}(-1)\), while on a quadric fiber \(Q\) it restricts to \(\cO_Q(-1)\). These line bundles are acyclic. For any quadric, this follows from
    \[0\to\cO_{\PP^3}(-3)\to\cO_{\PP^3}(-1)\to\cO_Q(-1)\to0.\]
    The description of bundles by \cite{Takagi1999} extends these calculations to every fiber, so cohomology and base change gives the asserted vanishing.

    Let \(\Phi\coloneq Ri_*Lg^*\). For \(F,G\in \Db(C)\), the Cartier self-intersection filtration has associated terms
    \[R\Hom(Lg^*F,Lg^*G),\qquad R\Hom(Lg^*F,Lg^*G\otimes\cO_E(E))[-1].\]
    The second term vanishes, while the projection formula identifies the first with \(R\Hom(F,G)\).
    Therefore \(\Phi\) is fully faithful. Since \(C\) and \(X\) are smooth and projective, the functor has both adjoints and its image is admissible.
\end{proof}

\subsection{The \texorpdfstring{\((3, 2)\)}{(3, 2)} case}
We now consider the \((3,2)\) case. Andreatta and Wi\'sniewski describe its isolated two-dimensional fibers in \cite{AW1998}. The equidimensional case and the divisorial models with smooth target are covered by \Cref{lem:smooth}. At a singular target point, the special fiber is \(F=\PP^2\) with
\[N^\vee_{F/X}=\bigl(T_{\PP^2}(-1)\oplus\cO_{\PP^2}(1)\bigr)/\cO_{\PP^2},\]
and its image is an ordinary double point of \(Y\). We construct the required admissible subcategory near each such point and then use descent for semiorthogonal decompositions to glue the local subcategories. This also allows several singular points of this type to occur on the same target.

Let \(F\) be such a special fiber. After shrinking \(Y\), assume that it is the only special fiber in the chosen neighborhood. Let \(P=f(F)\) be the ordinary double point. In this case \(S=f(E)\) is smooth. Consider
\[\begin{tikzcd}
	E & X \\
	S & Y
	\arrow["i", from=1-1, to=1-2]
	\arrow["g"', from=1-1, to=2-1]
	\arrow["f", from=1-2, to=2-2]
	\arrow["j"', from=2-1, to=2-2]
\end{tikzcd}\]

\begin{lemma}
The image of \(\Phi=Ri_*(Lg^*(-)\otimes\cO_E(E))\) is an admissible subcategory of \(\Db(X)\).
\end{lemma}

\begin{proof}
The geometric model gives \(Rg_*\cO_E\cong\cO_S\). Moreover,
\[\omega_X\cong f^*\omega_Y\otimes\cO_X(E).\]
The target is Gorenstein and has rational singularities. Hence Grauert--Riemenschneider vanishing applies. Grothendieck duality and the projection formula then give
\[Rf_*\cO_X(E)\cong\cO_Y.\]
Together with \(Rf_*\cO_X\cong\cO_Y\), the exact sequence
\[0 \to \cO_X \to \cO_X(E) \to \cO_E(E) \to 0\]
gives \(Rg_*\cO_E(E)=0\). The Cartier self-intersection filtration and the projection formula now show that \(\Phi\) is fully faithful. Since \(S\) and \(X\) are smooth, \(i\) is a Cartier immersion, and \(g\) is proper, the adjoints constructed by duality preserve perfect complexes; hence the image is admissible.
\end{proof}

Let \(\cA=\im\Phi\) and \(\cD={}^\perp\cA\). Since \(Rf_*\Phi=0\), the usual pullback--pushforward adjunction shows that \(Lf^*\Perf(Y)\subseteq\cD\). For \(G\in\Perf(Y)\), we have \(f^!G\cong Lf^*G\otimes\cO_X(E)\). Tensoring the triangle
\[\cO_X \to \cO_X(E) \to \cO_E(E) \to,\]
we obtain a triangle
\[Lf^*G \to f^!G \to \Phi(Lj^*G)\]
whose last term belongs to \(\cA\). For \(\cF\in\cD\), applying \(R\Hom_X(\cF,-)\) to this triangle and using Grothendieck duality gives the second of the following adjunction identities \(R\Hom_X(\cF,Lf^*G)\cong R\Hom_Y(Rf_*\cF,G)\).
Here the identities are asserted for \(G\in\Perf(Y)\). Moreover, \(Rf_*Lf^*G\cong G\) by the projection formula. These show that \(\cD\) is a weakly crepant local categorical resolution \cite{Kuz2008}.

The local decomposition is also linear over \(\Perf(Y)\). Indeed, for \(G\in\Perf(Y)\) and \(F_S\in\Perf(S)\), the projection formula gives
\[Lf^*G\otimes\Phi(F_S)\cong\Phi(Lj^*G\otimes F_S).\]
Thus \(\cA\) is stable under tensoring by pullbacks of perfect complexes on \(Y\). Since such complexes are dualizable, \(\cD={}^\perp\cA\) is stable under the same action, and the adjoints to the inclusions are linear over \(\Perf(Y)\). Consequently \(\Perf(X)=\langle\cA,\cD\rangle\) is a \(\Perf(Y)\)-linear semiorthogonal decomposition.

Choose an affine cover \(\{Y_i\}_{i\in I}\) such that each \(Y_i\) is smooth or is a neighborhood of a single ordinary double point as above, containing no other two-dimensional fiber. Arrange that overlaps between distinct charts lie in the smooth locus. Write \(X_i=f^{-1}(Y_i)\) and \(f_i=f|_{X_i}\). On a singular chart put \(S_i=f_i(E\cap X_i)\), denote the corresponding functor above by \(\Phi_i\), and take \(\cD_i={}^\perp\im\Phi_i\). The adjoints to \(\Phi_i\) preserve perfect complexes by the Cartier immersion formula and proper pushforward to the smooth variety \(S_i\). On a smooth chart take \(\cD_i=Lf_i^*\Perf(Y_i)\). Here proper pushforward to the smooth variety \(Y_i\) preserves perfect complexes, so the right adjoint is \(Rf_{i*}\), and the left adjoint is obtained by applying perfect duality on either side of \(Rf_{i*}\). These choices give \(\Perf(Y_i)\)-linear admissible decompositions of \(\Perf(X_i)\).

On an overlap \(U=Y_i\cap Y_j\) with \(Y_i\) a singular chart, every nontrivial fiber of \(f_U:X_U=f^{-1}(U)\to U\) is one-dimensional. Thus \(f_U\) is the blow-up of the smooth center \(S_U=S_i\cap U\). Orlov's formula \cite{Orlov1992} gives
\[\Perf(X_U)=\big\langle\Phi_U\Perf(S_U),Lf_U^*\Perf(U)\big\rangle,\]
where \(\Phi_U\) is the restriction of \(\Phi_i\). Hence the component from a singular chart restricts to \(Lf_U^*\Perf(U)\). On overlaps of smooth charts, the pullback descriptions agree directly. These equalities of full subcategories of \(\Perf(X_U)\) respect further restriction and hence satisfy the cocycle condition.

Applying \cite[Corollary 3.17]{AE2021} to the natural enhancements of these compatible linear decompositions produces an admissible subcategory \(\cD\subset\Perf(X)=\Db(X)\) with the specified local restrictions. It contains \(\cO_X\), while its right orthogonal restricts on any singular chart to the nonzero category \(\Phi_i\Perf(S_i)\). Thus both components are nonzero, and the decomposition is nontrivial.

\subsection*{Completion of the proof}

We can now complete the fourfold argument. If the contraction is birational, the cases treated at the beginning of this section give a nonzero proper admissible subcategory. Suppose it is of fiber type. The cases with smooth target are covered by \Cref{lem:smooth}. A B0 fiber carries an exceptional plane. A C3 fiber with \(m=3\), a C5 fiber, or a \(\mathrm{C}'\) fiber has exceptional structure sheaf by \Cref{thm:c3-exceptional,thm:c5,thm:three-planes}. If none of these fibers occurs, the special fibers are C2 fibers or cases with smooth target at the image point, including C3 with \(m=2\). The C2 theorem proves smoothness at the remaining special points. Thus the target is smooth everywhere, and \Cref{lem:smooth} applies. This completes the proof of \Cref{thm:main} in dimension four.

\begin{appendices}
\appendix
\section{An example of a C5 fiber}
\label{app:c5}

Neither \cite{Kachi1997} nor \cite{AW1998} gives an explicit construction of a C5 fiber. We record one here.

Put \(Y=\bbA_{u,v,w}^3\), \(P=\PP^4_{[a:b:c:d:e]}\times Y\), and \(Q=c^2+e^2\).
Let
  \[M =
    \begin{pmatrix}
      0 & Q & a & b & ue \\
      -Q & 0 & b & c & d \\
      -a & -b & 0 & u & v \\
      -b & -c & -u & 0 & w \\
      -ue & -d & -v & -w & 0 \\
    \end{pmatrix},
  \]
and \(X = V(\mathrm{Pf}_{4}(M)) \subset P\), whose ideal is generated by five principal \(4 \times 4\) Pfaffians, \(l_{1}, -l_{2}, q_{1}, -q_{2}, q_{3}\), where
\[l_1 = du + bw - cv,\]
\[l_2 = aw - bv + u^2e,\]
\[q_1 = wQ - bd + uec,\]
\[q_2 = vQ - ad + ueb,\]
\[q_3 = uQ - ac + b^2.\]

Let \(f: X \to Y\) be the projection. Then
\begin{proposition}
    The variety \(X\) is a smooth irreducible fourfold, and \(f\) is a projective \(K_X\)-negative elementary contraction with
    \[\rho(X/Y) = 1, -K_X \cong \cO_X(1).\]
    Its fiber over the origin is the reduced union
    \[f^{-1}(0) = \PP^2 \cup_L S_2\]
    of type C5. Every other fiber is a plane conic and the general fiber is smooth.
\end{proposition}

\begin{proof}
  We first show that the fiber over the origin is of type C5. When \(u = v = w = 0\), the five equations reduce to
  \[ad = bd = ac - b^{2} = 0.\]
  This ideal decomposes as \((ad,bd,ac-b^{2})=(a,b)\cap(d,ac-b^{2})\). Therefore \(f^{-1}(0) = A \cup B\), where \(A = \PP^{2}\) and \(B = V(d, ac - b^{2}) \cong \PP(1, 1, 2)\). The intersection is \(L = V(a,b, d) \cong \PP^{1}\).
  The general fibers are plane conics. For example, for \(u \neq 0\), eliminate \(d, e\), the remaining equation is
  \[
    G_{u} = u^{4}c^{2} + (bv - aw)^{2} - u^{3}ac + u^{3}b^{2} = 0.
  \]
  If \(w \neq 0\), then eliminate \(a, b\), we have
  \[
    G_{w} = w^{2}(c^{2} + e^{2}) - cvd + ud^{2} + uwce = 0.
  \]
  And if \(u = w = 0\), then \(v \neq 0\) and \(b = c = 0\). The equation is
  \[ve^{2} - ad = 0.\]

  Now we prove that \(X\) is smooth. We have the following \(\C^{*}\)-action on \(X\),
  \begin{equation}
    \label{eq:c-action}
    t \cdot ([a:b:c:d:e], u,v,w) = ([t^{2}a:tb:c:td:e], t^{2}u, t^{3}v, t^{2}w).
  \end{equation}
  We claim that it suffices to prove the smoothness at the fixed points of the action. The fixed locus is
  \[F = L \cup \{p_{a}, p_{d}\},\]
  where \(p_{a} = [1:0:0:0:0]\), \(p_{d} = [0:0:0:1:0]\) over the origin. Every point \(x \in X\) has a limit in \(F\) when \(t \to 0\). Since the singular locus is closed and invariant under \(\C^{*}\)-action, if \(X\) has a singular point, it must have a singular point in \(F\). This proves the claim.

  It remains to prove smoothness on \(F\). Let \(p = ([0:0:c:0:e], 0) \in L\). If \(Q(p)\neq0\), then
  \[I_{X} = (q_{1}, q_{2}, q_{3})\]
  near \(p\), and the Jacobian with respect to the base coordinates \((u,v,w)\) has determinant \(-Q(p)^3\neq0\).
  If \(Q(p)=0\), then \(c\neq0\). On the affine chart \(c=1\), these are the two points with \(e=\pm\ii\). Since \(c\) is a unit, the local ideal is \(I_X=(l_1,q_1,q_3)\), and
  \[\left.\frac{\partial(l_1,q_1,q_3)}{\partial(v,u,a)}\right|_p
    =\begin{pmatrix}-1&0&0\\0&e&0\\0&0&-1\end{pmatrix}.\]
  Its determinant is \(e=\pm\ii\neq0\). Thus \(X\) is smooth at every point of \(L\). Similar computations prove smoothness at \(p_a\) and \(p_d\).

  Since the general fiber is a smooth conic, exactly one irreducible component of \(X\) dominates \(Y\), and every other component maps into a proper closed subset of \(Y\) and therefore has dimension at most three. As \(X\) is smooth of pure dimension four, no such component can exist. Hence \(X\) is irreducible.

  We now prove that \(-K_{X}\) is \(f\)-ample. By the computation above, we have \(\dim X = 4\) and \(\mathrm{codim}_{P}X = 3\). Let \(a,b,c,d,e\) have degree \(1\) and \(u, v, w\) have degree zero. The five Pfaffians have degree \((1, 1, 2, 2, 2)\). The Buchsbaum-Eisenbud resolution gives
  \[0 \to \cO_{P}(-4) \to \cO_{P}(-3)^{\oplus 2} \oplus \cO_{P}(-2)^{\oplus 3} \to \cO_{P}(-1)^{\oplus 2} \oplus \cO_{P}(-2)^{\oplus 3} \to \cO_{P} \to \cO_{X} \to 0,\]
  where \(\cO_{P}(t)\) denotes the pullback of \(\cO_{\cP^{4}}(t)\). Since \(\omega_{p} \cong \cO_{P}(-5)\), the duality gives
  \[\omega_{X} \cong \mathcal{E}\mathrm{xt}^{3}_{P}(\cO_{X}, \omega_{P}) \cong \cO_{X}(-1).\]
  Thus \(-K_{X} \cong \cO_{X}(1)\). And the restriction of polarization satisfies
  \[\cO_{X}(1)|_{A} \cong \cO_{\PP^{2}}(1),\]
  \[\cO_{X}(1)|_{B} \cong \cO_{\PP(1,1,2)}(2).\]
  Therefore the special fiber is of type C5.

  Finally, we prove that \(\rho(X/Y) = 1\). Write \(H = \cO_{X}(1)\). We claim that every Cartier divisor \(D\) on \(X\) is numerically proportional to \(H\) on curves contracted by \(f\). Recall \(f^{-1}(0) = E = A \cup_{L} B\). \(\Pic_{A} = \Z[H|_{A}]\) and \(\Pic(B) = \Z[H|_{B}]\). Hence \(\Pic_{E} = \Z [H|_{E}]\). Consequently for every curve \(C \subset E\), we have
  \[D \cdot C = mH \cdot C,\]
  for some \(m \in \Z\). Now let \(C \subset X\) be any irreducible curve contracted by \(f\) not at the origin. Consider the \(\C^{*}\) action \eqref{eq:c-action} again. Its orbit \(t \cdot C\) defines a family of curves in \(\PP^{4}\) with constant Hilbert polynomial. By properness of the Hilbert scheme \cite[\href{https://stacks.math.columbia.edu/tag/0DPH}{Lemma 0DPH}]{stacks-project}, this family has a flat limit \(C_{0}\) as \(t \to 0\). The limit is supported on \(E = f^{-1}(0)\). As intersection number are constant in this family, we have
  \[D \cdot C = D \cdot C_{0},\quad H \cdot C = H \cdot C_{0}.\]
  But \(D|_{E}\) and \(mH|_{E}\) are numerically equivalent. Therefore \(D \cdot C = mH \cdot C\). Hence \(\rho(X/Y) = 1\).
\end{proof}

\end{appendices}

\bibliographystyle{alpha}
\bibliography{sample}

\end{document}